\documentclass{article}

\usepackage{arxiv}

\usepackage[utf8]{inputenc} 
\usepackage[T1]{fontenc}    
\usepackage{hyperref}       
\usepackage{url}            
\usepackage{booktabs}       
\usepackage{amsfonts}       
\usepackage{nicefrac}       
\usepackage{microtype}      
\usepackage{lipsum}

\usepackage{amsmath,amsfonts,amssymb,amsthm}
\usepackage{graphicx,verbatim,longtable,etoolbox}
\usepackage{caption, subfig}
\usepackage[table,xcdraw]{xcolor}

\newcommand{\citea}{\cite}
\def\bm{\boldsymbol}
\def\Pl{\mathrm{Pl}}
\def\pl{\mathrm{pl}}

\DeclareMathOperator*{\argmax}{arg\,max}

\usepackage[greek,english]{babel}
\usepackage{newunicodechar}
\newunicodechar{ℝ}{\mathbb{R}}
\newunicodechar{θ}{\vartheta}
\newunicodechar{α}{\alpha}
\newunicodechar{β}{\beta}
\newunicodechar{Θ}{\Theta}
\newunicodechar{ψ}{\psi}
\newunicodechar{λ}{\lambda}
\newunicodechar{Π}{\Pi}
\newunicodechar{ε}{\epsilon}
\newunicodechar{σ}{\sigma}

\newunicodechar{∈}{\in}
\newunicodechar{⇒}{\Rightarrow}
\newunicodechar{→}{\to}
\newunicodechar{•}{\bullet}
\newunicodechar{≤}{\le}
\newunicodechar{≥}{\ge}
\newunicodechar{∀}{\forall}

\title{Exact model comparisons in the plausibility framework}
\author{
	Stefan Böhringer$^*$\\
		Department of Biomedical Data Sciences\\
		Leiden University Medical Center\\
		Einthovenweg 20, 2333 ZC Leiden, The Netherlands\\
		\texttt{s.boehringer@lumc.nl} \\
		$^*$corresponding author\\
   \And
	Dietmar Lohmann\\
		Institut für Humangenetik\\
		Universitätsklinikum Essen\\
		Essen, Germany \\
		\texttt{dietmar.lohmann@uni-due.de} \\
}

\begin{document}

\maketitle


\begin{abstract}
Plausibility is a formalization of exact tests for parametric models and generalizes procedures such as Fisher's exact test. The resulting tests are based on cumulative probabilities of the probability density function and evaluate consistency with a parametric family while providing exact control of the $\alpha$ level for finite sample size. Model comparisons are inefficient in this approach. We generalize plausibility by incorporating weighing which allows to perform model comparisons. We show that one weighing scheme is asymptotically equivalent to the likelihood ratio test (LRT) and has finite sample guarantees for the test size under the null hypothesis unlike the LRT. We confirm theoretical properties in simulations that mimic the data set of our data application. We apply the method to a retinoblastoma data set and demonstrate a parent-of-origin effect.\par
Weighted plausibility also has applications in high-dimensional data analysis and P-values for penalized regression models can be derived. We demonstrate superior performance as compared to a data-splitting procedure in a simulation study. We apply weighted plausibility to a high-dimensional gene expression, case-control prostate cancer data set.\par
We discuss the flexibility of the approach by relating weighted plausibility to targeted learning, the bootstrap, and sparsity selection.
\end{abstract}


{\bf Declaration of interest:} The authors declare no conflict of intersts.\par
{\em CRediT author statement:} {\bf Stefan Böhringer:} Conceptualization, Methodology, Software, Data curation, Writing- Original draft preparation, {\bf Dietmar Lohmann:} Conceptualization, Investigation, Data curation, Writing- Reviewing and Editing.

\section{Introduction}

Exact inference has a long tradition in the statistical literature. Well known examples include Fisher's exact test \citea{fisher_interpretation_1922, agresti_survey_1992} and Clopper-Pearson intervals for the binomial distribution \citea{clopper_use_1934, agresti_approximate_1998}. Reasons for pursuing exact inference are warranted in cases when reliance on asymptotic properties of statistical procedures such as maximum-likelihood (ML) seem doubtful which include small sample size, complex models and skewed distributions. Plausibility is a framework that allows to compute exact P-values in a likelihood framework \citea{martin_plausibility_2015} and exploits the concept of using cumulative probabilities for statistical inference. Under a parametric model, the  cumulative probability of observed data maximized over the parametric family is considered and measures consistency of the data with this family. This can be interpreted as a goodness-of-fit statistic for this family. The statistic itself can be used as a P-value. Plausibility has been demonstrated theoretically and in simulations to have exact properties \citea{martin_plausibility_2015} and includes the examples mentioned above as special cases.\par
The focus on a single family of distributions can be a major limitation in some applications. For example, in a regression setting, the effect of a specific covariate might be of interest, controlled for a number of nuisance covariates. In this case, the plausibility statistic might have low power as it rejects against all alternatives deviating from the null distribution in contrast with a model comparison of two nested models. The motivating example for this study concerns retinoblastoma (RB) patients and implies research questions best answered with model comparisons. RB is a hereditary tumor syndrome, where a pre-existing variant allele increases the risk of tumor formation in the eye. A single variant allele segregates in a family and the risk increase due to this allele (penetrance) is of interest. The so-called Knudson model was a first statistical approach based on a Binomial model \citea{knudson_mutation_1971}. As extensions, the analysis of different effects of different mutations and the effect of the parental origin (parent-of-origin; POO) are of interest. Families can be small and exact inference seems prudent. In principle, these research questions can be analyzed using a goodness-of-fit approach. For example, using equal penetrances of RB families for our family of distributions should reject data coming from inhomogeneous families for sufficient sample size with high probability. On the other hand, this approach wastes power as the test would also reject alternatives which are not of interest, {\it e.g.} inhomogeneity within families. One major goal of this study is therefore to extend the plausibility framework with the possibility to perform model comparisons which focuses power on certain alternatives. This will be achieved by introducing a weighing scheme leading to weighted plausibility which puts probability mass on the model comparison rather than the goodness-of-fit.\par

Model comparisons in the plausibility framework are straightforward to extend to high-dimensio\-nal data analysis. Penalized regression is a widely used method for such data sets, however, the derivation of P-values is challenging. Data splitting is one possible approach, which selects variables in one part and estimates effect sizes in the second part \citea{wasserman_high_2009,meinshausen_p-values_2009}. We contrast this approach with a plausibility approach. The Bayesian interpretation of penalized regression allows to perform high-dimensional inference\citea{martin_plausibility_2015}. If only the alternative is high-dimensional, a strictly frequentist analysis is possible.\par

The paper is structured as follows: First, in section 2, we re-state the plausibility model and extend it with a weighing scheme that allows for model comparison. We show close kinship to likelihood procedures. Section 3 contains results for the normal model and introduces a global testing procedure. and demonstrates applications to high-dimensional data. In section 4, we apply the method to a retinoblastoma (RB) data set, and a high-dimensional prostate cancer data set. This section also contains simulations to systematically evaluate the procedures. Finally, we end with a discussion where we highlight some future directions, mention limitations and discuss the relationship with targeted learning. An appendix contains proofs.

\section{Methods}

\subsection{Plausibility functions}

We start this section by quickly reiterating important definitions and results from the plausibility framework. Results are taken from previous work unless stated otherwise\citea{martin_plausibility_2015}. We add an asymptotic result at the end of the section. We assume data $\mathbf{Y}$ to be sampled from a member of a parametric family of distributions $P_\theta$. First, we define statistic $T$ as

\begin{align*}
	T_{Y, \theta} = T_{Y, \theta, l} = \exp \{ -(l(Y, \theta) - c(Y)) \}.
\end{align*}

Here, $l$ is a loss function, in the following taken to be the negative log-likelihood, and $c(Y)$ is a normalizing term, usually taken to be $c(Y) = l(Y, \hat \theta)$ where $\hat \theta$ is the maximum likelihood estimator (MLE). The normalizing term $l(Y, \hat \theta)$ allows to develop the theory by guaranteeing that statistic $T_{Y, \theta}$ has support $[0, 1]$ for any $P_\theta, \theta \in \Theta$. However, this is non-essential. We also consider $c(Y) = 0$ later. The plausibility function is defined as:
$$
	\mathrm{pl}_Y(A) = \sup_{\theta \in A} F_\theta( T_{Y, \theta} ),
$$
where $F_\theta$ is the distribution function of $T_{Y, \theta}$. We call $\theta^* = \theta^*(T_{Y, \theta}) :=\arg\sup _{\theta \in A} F_\theta( T_{Y, \theta} )$ the plausibility estimate, when it exists.
To shorten notation, we define the distribution function of $\mathrm{pl}_Y(A)$ as
$
	\mathrm{Pl}_{Y,\theta}(\alpha) := \Pl_\theta(\alpha) := P_\theta(\pl_Y(A) \le \alpha)
$.
The reference to $Y$ is omitted when it is clear which plausibility function is used. We also abbreviate $\mathrm{pl}_Y(\theta) := \mathrm{pl}_Y(\{ \theta \})$ for $\theta \in \Theta$.

\newtheorem{theorem:larger}[theorems]{Theorem}
\begin{theorem:larger}[Martin 2015]\label{theorem:larger}
	Let $A \subset \Theta$. For any $\theta \in A$, $\alpha \in [0, 1]$, $Y \sim P_\theta$, $\mathrm{pl}_Y(A)$ is stochastically larger than uniform, {\it i.e.}
	$$
		\sup_{\theta \in A} \mathrm{Pl}_{\theta}(\alpha) \le \alpha
	$$
\end{theorem:larger}

We now assume that $Y = (Y_1, ..., Y_n)$ is an i.i.d. sample with $Y_1 \sim P_\theta$ and denote the plausibility function and its CDF with $\mathrm{pl}_n$ and $\mathrm{Pl}_n$, respectively, to indicate sample size. If $l$ does not have discontinuities, $\mathrm{pl}_n$ is uniformly distributed for all $n$ on $(0, 1)$ as shown previously. Otherwise convergence holds in distribution as $n \to \infty$ when the following uniqueness condition is met.

\newtheorem{definition:unique}[definition]{Definition}
\begin{definition:unique}\label{definition:unique}
	Likelihood $L(y ; \theta)$ has unique point masses if and only if for point mass $\alpha_m$, the set
	$Y_m := \{y \in \mathbb R^n | \pl_{y, \theta}(A) = \alpha_m\}$ only contains identical observations up to  exchangeability, {\it i.e.} observations only differ up to ordering $y, y' \in Y_m \Rightarrow
	(y_{(1)}, ..., y_{(n)}) = (y'_{(1)}, ..., y'_{(n)})$,
	where $x_{(i)}$ denotes the $i$th order statistic for vector $x = (x_1, ..., x_n)$.
\end{definition:unique}

\newtheorem{lemma:conv}[lemmas]{Lemma}
\begin{lemma:conv}\label{lemma:conv}
	If $L$ has unique point masses and under the other assumptions of the previous paragraph,$\mathrm{pl}_n(\theta)$ converges weakly to the standard uniform $U(0, 1)$.
\end{lemma:conv}

The proof is given in the appendix.

The restriction to unique point masses guarantees the uniqueness of the plausibility estimate and is theoretically strong but usually not strong in practice as in the following example. To illustrate the problems that might occur, consider a likelihood of i.i.d. Bernoulli variables for which the likelihood is modelled as $L(\theta; Y) = \prod_{i = 1}^N \theta^{Y_i} (1 - \theta)^{1 - Y_i}$. For $\theta = .5$, every data set has the same probability and therefore both $\theta = .5$ and the MLE $\theta = \hat \theta$ maximize the plausibility function for every data set. We call a value that maximizes the plausibility for every data set a non-plausible value. The uniqueness condition is usually not fulfilled for most discrete distributions. For example, the binomial distribution with parameter 0.5 would be non-unique due to symmetry around 0.5. Convergence to the uniform would still hold as non-uniqueness would be restricted to to a small subset of events ({\it i.e.} a pair, see appendix \ref{app:nonplausible}). We do not try to optimally characterize conditions on the likelihood to guarantee unique estimates. Instead, we see lemma \ref{lemma:conv} as a guiding principle. For example, a solution for the Bernoulli example is to add the binomial coefficient to the likelihood which guarantees unique estimates (up to symmetries). More details are given in the appendix (section \ref{app:nonplausible}).

\subsection{Plausible model comparisons}\label{section:comparison}

We prepare model comparison by considering a real-valued, measurable function $w$ that acts on realizations of $Y$. We assume $w: \mathbb{Y} \to \mathbb{R}$ to be free of $\theta$. We first observe that when defining $T$, we can construct plausibility functions based on $w(Y)$ by replacing the exponentiated loss function $exp \circ -l$ by $w$ to get
$T^w_{y, \theta} := w(y)/c^w(y)$. If normalization is desired $c^w(y)$ is taken to be $c^w(y) = \sup_{y}w(y)$, or 1 otherwise. The distribution function of $T^w$ can then be written as
$F_{\theta}(t) = P_\theta(\{y|w(y) \le t\})$, 
which induces $\pl^w_y(A) = \sup_{\theta \in A} F_\theta( T^w_{y, \theta} )$ and $\Pl^w_{\theta}(\alpha) = P_\theta(\pl^w_y(A) < \alpha)$. If $w$ is bijective, a strict order is imposed on events and the CDF is calculated under this ordering.\par
This concept generalizes standard or unweighted plausibility introduced above, as $w(y) := pl_y(A)$ leads to $\pl_y^w(A) = \pl_y(A)$ (see appendix \ref{app:specialcase}). Note that $w$ is a fixed function independent of $\theta$ as it does not depend on data $Y$, {\it i.e.} it can be pre-computed for each realization $y$ without seeing the data. This definition merely serves to show the generalization, it is not helpful for performing the computations in unweighted plausibility.

\newtheorem{lemma:W}[lemmas]{Lemma}
\begin{lemma:W}\label{lemma:W}
	With the notation from the previous paragraph, let $\theta \in \Theta, A \subset \Theta$ and $w$ some test statistic $w: \mathbb{Y} \to \mathbb{R}$ which is free of $\theta$. Then,
	$$
		\sup_{\theta \in A} \Pl^w_{\theta}(\alpha) \le \alpha,
	$$
	{\it i.e.} $\pl^w_Y(A)$ is stochastically larger than uniform, with $\Pl^w_{\theta}(\alpha) = P_\theta( F_\theta( T^w_{y, \theta} ) \le \alpha)$. 
\end{lemma:W}
\begin{proof}
	By definition of $\Pl^w_{\theta}$,
	\begin{align*}
		\Pl^w_{\theta}(\alpha)
			& = P_\theta(\sup_{\theta \in A} F_\theta( T^w_{Y, \theta} ) \le \alpha)\\
			& \le P_\theta(F_\theta( T^w_{Y, \theta} ) \le \alpha).
	\end{align*}
	As $F_\theta$ is the distribution function of $T^w_{Y, \theta}$,  $P_\theta(F_\theta( T^w_{Y, \theta} ) \le \alpha) \le \alpha$ by definition. Supremizing over $\theta$ completes the proof.
\end{proof}

This proof implies Theorem \ref{theorem:larger} as a special case. We now consider model comparisons. The idea is to choose a weighing function $w$ such that a model comparison is performed. Let the null hypothesis be represented by $\Theta_0 \subset \Theta$ and the alternative by $\Theta_1 \subset \Theta$ with $\Theta_0 \subset \Theta_1$, {\it i.e.} the situation of nested model comparisons. We now define a weighing function $w(y) = \sup_{\theta \in \Theta_0} l(y, \theta) - \sup_{\theta \in \Theta_1} l(y, \theta) =: l(y, \hat \theta_0) - l(y, \hat \theta_1) $.

\newtheorem{remark:isTest}[remarks]{Remark}
\begin{remark:isTest}\label{remark:isTest}
	With the notation above, $w(y) = l(y, \hat \theta_0) - l(y, \hat \theta_1)$,
	$$
		U := \pl_{Y}^w(\Theta_0)
	$$
	defines a testing procedure with rejection region $[0, \alpha)$ for alpha level $\alpha$.
\end{remark:isTest}

For discrete distributions, $U$ is a cumulative sum of probabilities. In these cases, the cumulative sum proceeds by summing event probabilities that are likely under the alternative but whose probabilities are evaluated under the null hypothesis. Intuitively, $U$ will therefore be more likely to reject the null if the observation was indeed drawn under the alternative. To formally characterize $U$, we now show that it is asymptotically equivalent to the likelihood ratio (LR) test for the same model comparison. $U$ can therefore be considered an exact version of the LR test. We base our argument on the comparison of rejection regions of the LR and weighted plausibility tests.

\newtheorem{theorem2}[theorems]{Theorem}
\begin{theorem2}\label{theorem:lrEquiv}
	Let $\theta \in \Theta_0$ be fixed and known, $\alpha \in (0, 1)$. For $w(y) := l(y, \hat \theta_0)/ l(y, \hat \theta_1)$,
	The rejection region for $U = \pl_{Y}^w(\Theta_0),$ is asymptotically identical to the rejection region of the LR-test, {\it i.e.} the probability mass of the symmetric difference between the rejection regions converges to 0 in probability.
\end{theorem2}

We give the proof in the appendix. From this asymptotic equivalence, some properties of the LR procedure are inherited by the weighted plausibility test.

\newtheorem{corollary:efficient}[corollaries]{Corollary}
\begin{corollary:efficient}\label{corollary:efficient}
	In the one-parameter situation, {\it i.e.} $\dim(\theta) = 1$, the test $U = \pl^w_{Y}(\Theta_0)$ is asymptotically efficient.
\end{corollary:efficient}

In the next section, we deal with the problem of constructing confidence regions for the unrestricted parameter $\theta \in \Theta_1$.

\subsection{Marginal Plausibility Functions}

In the context of nested model comparisons, often it is possible to express the null and alternative hypotheses by splitting the parameter vector $\theta = (\psi, \lambda) \in \Psi \times \Lambda = \Theta$ and constraining $\psi$ to a subset $\Theta_0^\psi \subset \Theta$ under the null, while leaving $\lambda$ free. $\lambda$ can be seen as a nuisance parameter. In this context, it is interesting to consider the relative profile likelihood to obtain

\begin{align*}
	T_{Y, \psi} = L_Y( (\psi, \hat \lambda(\psi)) )/c(Y),
\end{align*}

which allows to define the so-called marginal plausibility function:

\begin{align*}
	\mathrm{mpl}_Y(A) = \sup_{\psi \in A} F_\psi(T_{Y, \psi}).
\end{align*}

In principle, it is possible to base inference on the plausibility region for $\psi$:

\begin{align*}
	\Pi_Y(\psi) := \{ \psi : \mathrm{mpl}_Y(\psi) > \alpha \}.
\end{align*}

However, in order to be exact, the distribution of $T_{Y, \psi}$ has to be free of $\lambda$. This means that $T_{Y, \psi}$ has to be an ancillary statistics of $\lambda$ which is a strong limitation in practice. The reason is that  $T_{Y, \psi}$ will be evaluated in $(\psi, \hat \lambda(\psi))$ instead of the true $(\psi, \lambda(\psi))$. To make coverage exact, the distribution of $\hat \lambda(\psi)$ would have to be known, which is difficult in practice.\par

To prepare an alternative approach, we first introduce the profile plausibility which is constructed analogously to the profile likelihood. If $\theta$ is partitioned as above and $A = \{ \psi \} \times A_\lambda$ for some $\psi \in \Psi$, we call $\mathrm{ppl}_Y(\psi) = \sup_{\lambda \in A_\lambda} F_{(\psi, \lambda)}(T_{Y, (\psi, \lambda)})$ the profile plausibility function of likelihood $L_Y((\psi, \lambda))$ for $\psi$.

\newtheorem{remark:profilePlausibility}[remarks]{Remark}
\begin{remark:profilePlausibility}\label{remark:profilePlausibility}
	For likelihood $L_Y((\psi, \lambda))$, $A = \{(\psi, \lambda) \in \Psi \times \Lambda \}$, the profile plausibility estimate $(\psi^*, \lambda^*)$ is given by $(\arg \sup_\psi \mathrm{ppl}_Y(\psi), \arg \sup_\lambda T'_{Y, \lambda})$, with $T'_{y, \lambda} := T_{Y, (\psi^*, \lambda)}$.
\end{remark:profilePlausibility}


We now use the weighted plausibility framework to construct an alternative marginal plausibility function. The weighing function is parametrized by $\psi$ and statistic $T$ therefore becomes a function of $\psi$ and $\lambda$. We use the following weighing function:

\begin{align*}
	w(\psi, y) = w_y(\psi) = l(y, (\psi, \lambda^*)) - l(y, (\psi^*, \lambda^*)),
\end{align*}

where $(\psi^*, \lambda^*)$ is the plausibility estimate for data $y$. For $T^w_{Y, \psi; \lambda}$ we use the weighted profile plausibility which is defined as above by using $T^w$ instead of $T$ and denoted with $\mathrm{ppl}^w_Y(\psi)$. To emphasize that $w$ is independent from any data generating parameter, we use $(\psi^\bullet, \lambda^\bullet)$ for data generating parameters and $\psi, \lambda$ as function parameters.


We now use the profile plausibiliity function for $\psi$ to construct a weighted marginal plausibility region.

\begin{align}\label{form:marginal}
	\mathrm{mpl}^{w}_Y(A) = \sup_{\psi  \in A} \mathrm{ppl}^w_Y(\psi).
\end{align}

This leads to the weighted marginal plausibility region:

\begin{align}\label{form:marginal}
	\Pi^{w}_Y(\alpha) = \{ \psi \, : \, \mathrm{ppl}^w_Y(\psi) > \alpha \}.
\end{align}

With weighing function $w$ defined as above, the following lemma holds.

\newtheorem{lemma:marginal}[lemmas]{Lemma}
\begin{lemma:marginal}\label{lemma:marginal}
	The coverage probability of the weighted marginal likelihood is nominal for $\psi$, {\it i.e.} for $\alpha \in (0, 1), \theta^\bullet = ( \psi^\bullet, \lambda^\bullet )$, 
	\begin{align*}
		P_{\theta^\bullet}(
			 \Pi^{w}_Y(\alpha) \times \{ \lambda^\bullet\} \ni \theta^\bullet
		) \ge 1 - \alpha.
	\end{align*}
\end{lemma:marginal}
\begin{proof}
	We here give an informal argument and refer to the appendix for the formal proof. For each fixed $\psi$, the weighted marginal plausibility can be interpreted as a weighted plausibility function and is therefore statistically larger than uniform. The special form of the weighting function, which is the likelihood ratio evaluated in the plausibility estimates of the free parameters, guarantees that a value of $\psi$ is selected into the plausibility region if the observed data belongs to the probability $\alpha$ events close to the plausibility estimate. This argument is analogous to the equivalence of constructing confidence intervals and testing a point null-hypothesis.

\end{proof}


\section{The normal model}

Plausibility can be applied to high-dimensional data ($N < p$), {\it i.e.} data for which the number of predictors ($p$) exceeds that of observations ($N$). We first give a motivation by linear models and then introduce penalized models. For a linear model with data $\mathbf Y$, we assume a fixed design with design matrix $\mathbf X$ so that $\mathbf Y = \mathbf X \beta + \mathbf \epsilon$, with $\mathbf \epsilon = (\epsilon_1, ..., \epsilon_N),\,\, \epsilon_i \,  \stackrel{iid}{\sim}  \, N(0, \sigma^2)$. We take the estimate $\hat \theta$ for $(\beta, \sigma^2)$ as the ML-estimate, where the variance estimate is bias-corrected. The plausibility estimate of the parameter vector $\theta = (\beta, \sigma^2) \in A$ can then be found by:

$$
	\theta^* = (\beta^*, \sigma^{2*}) = \arg \sup_{\theta \in A} F_{(\beta, \sigma^2)}(T_{Y, (\beta, \sigma^2)}).
$$

\newtheorem{remark:nonPlausible}[remarks]{Remark}
\begin{remark:nonPlausible}\label{remark:nonPlausible}
	In the linear model above, the plausibility estimate does not exist for $A = \mathbb R^{p+1} \times \mathbb R^+$.
\end{remark:nonPlausible}

The remark follows from $F_{(\beta, \sigma^2)}(T_{y, (\beta, \sigma^2)}) \to 1$ for $\sigma^2 \to \infty$ and any fixed $\beta$ which in turn is due to the fact that any tail probability of the  normal distribution converges to 1 for increasing variance. We call such parameters non-plausible.\\
One possible approach to non-plausible parameters is to plug in an estimate of such parameters based on fixing the other parameters. For example, in the case of the linear model the unbiased estimate of residuals can be used ($\hat \sigma^2(\beta)$). We call a plausibility function based on such an estimate a ML profile-plausibility to indicate that a part of a partitioned parameter vector is fixed at the ML-estimate.

\newtheorem{lemma:profileML}[lemmas]{Lemma}
\begin{lemma:profileML}\label{lemma:profileML}
	In the linear model above, the profile-plausibility estimate for $\beta$ using the unbiased variance estimator for $\sigma^2$ coincides with the ML estimate, {\it i.e.} $\beta^* = \hat \beta$.
\end{lemma:profileML}
The proof uses elementary calculations and is given in the appendix.

\newtheorem{remark:degenerate1}[remarks]{Remark}
\begin{remark:degenerate1}\label{remark:degenerate1}
	In the linear model above, the profile-plausibility function is degenerate, {\it i.e.}
	$$\pl_{\mathbf Y}(A) = 1 \,\, \forall\, \mathbf Y,$$
	with $ A = \mathbb R^{p+1}$.
\end{remark:degenerate1}

The statement is due to the fact that each data set has a likelihood of $ (2 \pi \hat \sigma^2)^{-\frac{n}{2}} \exp(\frac{n - 1}{2}) $ when evaluated in $\hat \theta$. All potential data with different estimates, have lower likelihood as compared to the observed data.
We note, that conditional on any estimate $\theta^* = (\beta^*, \hat \sigma^2)$, data is uniformly distributed on the $S^{N - 1} \hat \sigma + \mathbf{X} \beta^*$ sphere. If data is standardized first, the uniformity is on $S^{N - 1}$ directly. In most situations the scale of the variable is not of interest or even arbitrary. In these cases it is justified to use the conditional distribution $\mathbf Y|\hat \sigma^2$ as the null distribution in the plausibility model, {\it i.e.} the density is that of a multivariate normal distribution with a fixed independent covariance matrix and the diagonal fixed at $\hat \sigma^2$. If we call the corresponding weighted profile-plausibility function $\pl^{w, \hat \sigma^2}_{Y}(\Theta_0)$, this observation motivates the following lemma, using notation from lemma \ref{lemma:W}. We have added $\hat \sigma^2$ in the superscript to emphasize that we use a ML profile-plausibility, otherwise it is a standard weighted plausibility function.

\newtheorem{lemma:profile}[lemmas]{Lemma}
\begin{lemma:profile}\label{lemma:profile} With the notational conventions from lemma \ref{lemma:W} and the paragraphs above,
	$
		\pl^{w, \hat \sigma^2}_{Y}(\Theta_0)
	$
	is stochastically larger than uniform, with $\Theta_0 = \{ \beta = 0 \}$. 
\end{lemma:profile}
Intuitively, the lemma follows from the fact that, conditional on the profile-plausibility estimate, data is uniformly distributed. The proof is given in the appendix.  Draws from this distribution for any $\beta$ can be made as an iid sample from an arbitrary normal distribution after which the sample is re-standardized. The same draw can be re-standardized to new paramter values, which is important when stochastic integration is used (see below).\par

\subsection{Global test}
To motivate a testing procedure, we first assume that $\mathbf Y$ has known variance $\sigma^2$. If we are interested in a global test without nuisance covariates, the null hypothesis of interest is $\beta = 0$ against $\beta \ne 0$. If the alternative $\beta_a \ne 0$ is known, the likelihood ratio (LR) test is defined as

\begin{align*}
	\Lambda = \prod_i\varphi(Y_i; (\beta_a^T \mathbf X)_i, \sigma^2)/\prod_i\varphi(Y_i; 0, \sigma^2),
\end{align*}

where $\varphi(\cdot ; \mu, \sigma^2)$ is the density of the normal distribution with parameters $(\mu, \sigma^2)$. $\Lambda > c$ is a uniformly most powerful test due to the Neyman-Pearson lemma for appropriate $c$. This property of the LR motivates the use of the following weighing statistic:

\begin{align}
	\Lambda(\mathbf X, \mathbf Y) = \prod_i\varphi(Y_i; (\hat \beta_{a0}^T \mathbf X_0  + \hat \beta_a^T \mathbf X_a)_i, \sigma^2)/\prod_i\varphi(Y_i; (\hat \beta_0^T \mathbf X_0)_i, \sigma^2),\label{formula:lr-highdim}
\end{align}

where $\mathbf X$ partitions into nuisance covariates and predictors $\mathbf X = (\mathbf X_0, \mathbf X_a)$. $(\hat \beta_0)$ denotes the the MLE under the null and $(\hat \beta_{a0}, \hat \beta_a)$ is the MLE under the alternative partitioned by the covariate components.

\subsection{High-dimensional data}\label{section:high-dim}

We now consider the situation where $N < p$, otherwise keeping the linear model from the previous section. First, we investigate the problem of testing the global null hypothesis $\beta = 0$ for the model $\mathbf Y = \mathbf X \beta + \mathbf \epsilon$ as introduced above, where we assume $\mathbf Y$ to be centered to justify $\beta = 0$. As the problem is ill-posed, one solution is to use penalized regression for the estimation of $\hat \beta$. As $\hat \beta$ is no longer a ML-estimate, standard likelihood theory no longer applies and the distribution of $\hat \beta$ has to be recovered by additional steps. One approach is to use data-splitting as reviewed in the introduction. We will use one implementation of data splitting to compare to a plausibility comparison \citea{yuan_lol:_2019}. For our weighted plausibility approach, we use the same weighting function (\ref{formula:lr-highdim}) as above where we plug in penalized estimates. We consider the Lasso \citea{tibshirani_regression_1996}, elastic net \citea{zou_regularization_2017}, and Ridge penalties \citea{hoerl_ridge_1970} as implemented in {\it glmnet} \citea{friedman_regularization_2010}. Under certain conditions, the penalized estimates converge to the true parameter values for a limiting process for which both $N$ and $p$ tend to infinity (see references given in \citea{friedman_elements_2001}). In the finite samples situtation, the LRT statistic using penalized estimates needs to separate the models well to be a useful weighing function in the plausbility approach. It is difficult to attain theoretical guarantees. In this paper, we rely on simulations to investigate properties of this approach.

\subsection{Evaluation of the plausibility function}\label{section:high-dim}

In almost all cases it is not possible to develop a closed-form formula for the plausibility function and numeric evaluation is required. Discrete distributions might allow exact evaluation for small data sets. In general, stochastic integration is needed.

To evaluate the plausibility function for point null-hypothesis $θ_0$, random draws of $ Y^{(j)}$ are drawn from the null model $ Y^{(j)} \sim  F_{θ_0}$. The plausibility function is then approximated by

\begin{align}
	\pl^{w}_{Y}(θ_0)
		\approx
		\frac{1}{M} \sum_j^M I\{w(Y) > w(Y^{(j)})\},
	\label{formula:stochastic-integration}
\end{align}

for $M$ approximation samples. In general, the supremum has to be taken over the full set $Θ_0$. In principle, for each evaluation $\theta \in \Theta_0$ a new sample can be taken. However, this adds additional sampling variation which can dominate function variation close to the supremum. To avoid this problem, the samples $(Y^{(j)})_j$ drawn from an initial $\theta_0$ is reused as follows:

\begin{align}
	T^w_{Y, \theta}
		\approx
		\frac{1}{M} \sum_j^M
			\frac{ L(Y, \theta)}{ L(Y, \theta_0) }
			I\{ w(Y) > w(Y^{(j)}) \}.
	\label{formula:stochastic-integration}
\end{align}

This can be seen as the evaluation of an integral using importance sampling. $F_{\theta_0}$ takes the role of the proposal distribution to evaluate $F_{\theta}$. This approximation is justified by re-writing $T^w_{Y, \theta}$ as follows:

\begin{align*}
	F_\theta(T^w_{Y, \theta})
	    = & \int 1_{A(Y)} dP_{\theta}
	    = \int_{A(Y)} dP_{\theta}
		= \int_{A(Y)} L(y, \theta)dy \\
		= & \int_{A(Y)} \frac{L(y, \theta)}{L(y, \theta_0)} L(y, \theta_0) dy
		=   \int \frac{L(Y, \theta)}{L(Y, \theta_0)} I\{ w(y) > w(Y) \} dP_{\theta_0},
\end{align*}

where $A(Y) = \{ y | w(y) < w(Y) \}$ and $1_X$ is the characteristic function of $X$.
We use a grid search around the MLE to find the supremum. Using R-function $\mathrm{optim}$ usually also works reliably. After finding the optimum, new samples can be drawn from the optimum and the procedure can be iterated. In practice, this turns out not to be necessary and for the analyses and simulations in this paper the optimization has not been iterated.\par
For the plausibility region, a grid search is performed to bracket the $\alpha$ level-set of the plausibility region. In our implementation, this search has exponential running time in the number of parameters as we consider a joint grid for all parameters centered around the MLE or plausibility estimate. The efficient construction of plausibility regions for many parameters needs to be addressed in future research.

\section{Data Examples}

\subsection{Retinoblastoma}

RB is a childhood tumor of the eye that follows a dominant inheritance pattern \citea{knudson_mutation_1971}. The disease has given rise to the so called two-hit hypothesis: a tumor suppressor gene needs to acquire two mutations to inactivate both copies available on autosomal chromosomes. In familial cases, one variant copy is inherited and only a second mutation is necessary to initiate tumor formation. If the probability for this second hit is high, most individuals inheriting the first mutation will develop a tumor and disease appears to be dominant, as expressed in the penetrance of the disease (disease probability, given presence of first mutation). In RB, families with reduced penetrance are known and one question is whether characteristics of the first variant introduced by a mutation in parents can explain this variation. As a second important question, RB1 gene has been shown to be imprinted at least in some constitutional cells, meaning that only one parental copy is preferentially active in these cells. This can lead to allelic imbalance of expression in cells showing RB imprinting. However, this has not been shown for the putative precursor cells of retinoblastoma as yet \citea{bremner_cancer:_2014}. A statistical analysis can help clarifying this question by analyzing the effect of parental origin on disease penetrance.

\subsubsection{The Knudson model}

Let $Y_i \in \{0, 1, 2\}$ denote the number of affected eyes in individual $i = 1, ..., N$, $Y = (Y_1, ..., Y_N)$. We assume

\begin{align*}
	Y_i & = I\{X_{il} > 1\} + I\{X_{ir} > 1\},\\
	X_{ij}\ \mathrm{iid} & \sim \mathrm{Poisson} (\lambda), i = 1, ..., N, j \in \{l, r\},
\end{align*}

where $X_{ij}$ is the number of tumors that individual $i$ has in eye $j$. We assume that the number of tumors is not known, only the presence of tumors is recorded, making $Y_i$ the sum of two indicator variables. Since all $X_{ij}$ are considered independent, $Y_i$ is a binomial $Y_i \sim \mathrm{Binom}(2, p)$ with MLE
$$
	\hat{p} = \frac{1}{2N} \sum_i Y_i = \frac{N_U + 2N_B}{2N},
$$
where $N_U$ and $N_B$ are the number of unilaterally and bilaterally affected individuals, respectively. As we have the relationship $\lambda = - \log(1 - p)$, we get
$
	\hat{\lambda} = - \log (1 - \hat p)
$

$\lambda$ is the average tumor count per eye. If measured per individual (as in the Knudson paper), we re-parametrize as $\lambda_I = 2 \lambda$ (called $m$ in the Knudson paper) and get the eye-distribution $(p_N, p_U, p_B) = ( (1 - p)^2, 2 p (1 - p), p^2)$ (none, unilateral, bilateral), with $p = 1 - \exp(- \lambda_I/2)$. 

To model covariates, we use a logistic model for $p$. For individual $i$, we define disease probability $p_i$ as follows:
$$
	\mathrm{logit} (p_i) = \bm \beta^T \bm{x_i},
$$
where $\bm x_i$ is the covariate vector of individual $i$ and $\bm \beta$ is the vector of regression coefficients. We assume $x_{i1} = 1$ for an intercept model. In our context, relevant covariates are family membership as a proxy for variant type and parental origin of the variant allele. An example pedigree is shown in figure \ref{fig:pedigree}. Note, that families are ascertained, {\it i.e.} at least one member is affected by RB. This fact can be modeled by an ascertainment correction in the likelihood. Founders, {\it i.e.} individuals without parents in the pedigree, are ignored as they may have acquired the mutation and, if so, may have it present in a mosaic state, {\it i.e.} the variant allele would be present in only part of the cells of the body. In our notation, we assume that founders have already been removed, {\it i.e.} $N$ represents the effective number of individuals in the pedigree.

\begin{figure}
	\begin{center}
	\includegraphics[height=.5\hsize]{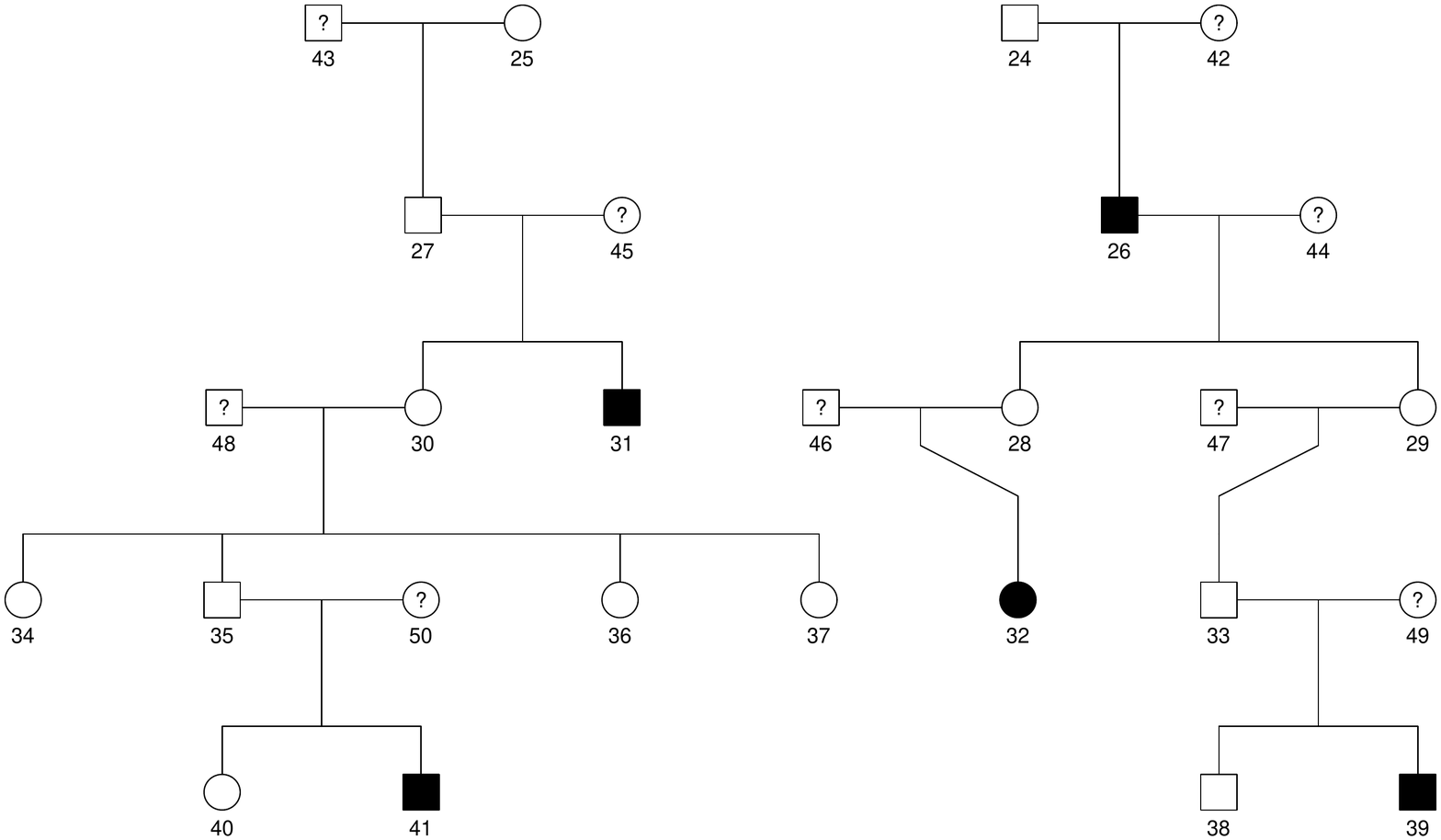}
	\end{center}
	\caption{Pedgree of a family with Retinoblastoma. Black: affected individuals, Question mark: unknown phenotype status.}\label{fig:pedigree}
\end{figure}

In total, the following likelihood is used:
\begin{align}
	L(Y; \bm \beta) = A(\bm \beta)^{-1} \prod_i^N \left\{ Z_i \left(
		\mathrm{logit}^{-1} (\beta^T \bm{x_i})
	\right) \right\}
	, \qquad
	Z_i(\pi) = P_\pi^B(Y_i = y_i) = {y_i \choose n} \pi^{y_i} (1 - \pi)^{n - y_i}
	\label{formula:lhRb}
\end{align}

where the ascertainment correction $A(\bm \beta) = 1 - P(Y = 0)^N = 1 - (\int_x P(Y = 0 | X = x) P(X = x) dx)^N$ represents the event that at least one individual is affected. If ascertainment is modeled, the covariate distribution needs to be modeled as well. In the following, we assume a random design, {\it i.e.} $X$ is drawn from the underlying population.

\subsubsection{Simulations}

All families in our data set contain several affected members ({\it e.g.} Figure \ref{fig:pedigree}). This implies that the ascertainment correction in formula (\ref{formula:lhRb}) will be close to one, {\it i.e.} the probability of no affected family members will be small. As a consequence, we did not model ascertainment in the simulations and used parameter values that emulate this characteristic. This makes it also easier to compare to other standard tests such a $\chi^2$ goodness-of-fit test which does not allow to account for ascertainment in its standard form.\par
For our simulations, we have implemented the calculation of $\pl_{y, w}$ by an exact computation, which was feasible for data considered here. An alternative is to use stochastic integration as mentioned above.\par
First, we regroup (\ref{formula:lhRb}) into sets with identical covariate vectors, by considering only discrete covariates.  Setting $c(y)$ to 0, $F_\beta(T^w_{Y, \beta})$ for a binomial $(k, n)$, covariate values $x_i$, and observed counts $y_i = (y_{i0}, ..., y_{in})$ for this covariate combination becomes

$$
	F_\beta \left( T^{w, x_i}_{Y, \beta} \right) = \sum_{e_i \in E_i(y_i, w)} P^M_{\pi(\beta)}(e_i) = \sum_{e_i \in E_i(y, w)} {|e_i| \choose e_i}\prod_j^n P_{\pi(\beta)}^B (Y = j)^{e_{ij}},
$$

where $E_i(y_i, w) = \{e_i \in \tilde{\Delta}^n | w(P^M_{\pi(\beta)}(e_i)) \le w(P^M_{\pi(\beta)}(y_i))\}$, where $\Delta^n_{\mathbb{N}}$ is the standard $(n - 1)$-simplex scaled to $n$ and restricted to $\mathbb{N}^n$ to represent all integer partitions of $n$ ($\tilde{\Delta}^n = n \Delta^n \cap \mathbb{N}^n$). For covriate combinations $x = (x_1, ..., x_K)$, we have

$$
	F_\beta \left( T^{w}_{Y, \beta} \right) = \sum_{e \in E(y, w)} \prod_i^K P^M_{\pi(\beta)}(e_i)
$$

For sample sizes up to $N = 30$, the plausibility function can still be efficiently evaluated exactly without resorting to stoachastic integration.\par

Families were simulated by deterministically distributing sample size across generations, adding a new founder per generation and drawing an inheritance vector for non-founders from a multivariate Bernoulli iid $\mathrm{Binom}(1, .5)$. Parent-of-origin was added as an additional covariate and inferred from the simulated data. Finally, outcome was drawn from the model specified above according to effects considered in the simulation scenarios. Sample size was set to $N = 8$ for all simulations, two families and two generations per family were simulated.\par
We compared the following procedures: (1) Unweighted plausibility ({\it Plausibility}), (2) Weighted plausibility with LR weighting ({\it W-Plausibility LR}), (3) Parametric Bootstrap ({\it Bootstrap}), (4) Pearson goodness-of-fit statistic comparing expected binomial counts under the logistic model with observed count ({\it Chi-Square}), (5) the likelihood-ratio test ({\it LR}), and finally (6) weighted plausibility with ''relative'' weights ({\it W-Plausibility Rel}). This relative plausibility is a weighted plausibility where parameters for the LR-weights are estimated from the data to be tested. Relative plausibility serves as an example where weights are not free of $\theta$. Unweighted plausibility evaluates a goodness-of-fit to a model where effects for variables of interest are set to zero (e.g. family).
For the simulations, 200 replications, and $10^3$ bootstrap samples (for procedure {\it Bootstrap}) were used. 

Figure \ref{fig:simN8} shows simulation results under the null hypothesis when an intercept model is compared to a model containing a family effect.
Unweighted plausibility is conservative, weighted plausibility precisely exhausts the $\alpha$-level up to a level of roughly 0.5, whereas relative plausibility is highly anti-conservative. Both the Pearon and the boostrap tests perform well but show $\alpha$-levels with conservative and anti-conservative behavior. The LR test is similar to the Bootstrap and Pearson tests except that deviations of size from $\alpha$-level are stronger.\par
Under the alternative, several scenarios with values for the intercept of 0.5 and 1 and family effects between 0.5 and 2 (on the log-OR scale) have been evaluated for an $\alpha$-level of 0.05 (figure \ref{fig:simN8A}). Relative plausibility performs best but has to be discounted due to anti-conservative behavior under the null-hypothesis. Otherwise, the LR test performs best but very similar to weighted plausibility. The difference is best explained by slightly anti-conservative behavior of the LR test at the 0.05 level. Bootstrap and Pearson's test show power close to the 0.05 level and seem unable to cope with the small sample size. Unweighted plausibility has some power when the intercept is small or when family effect is large (log-OR 2) but power is always smaller than 15\%. The simulations confirm that LR and weighted plausibility behave very similarly.

\begin{figure}
	\begin{center}
	\includegraphics[height=.45\hsize]{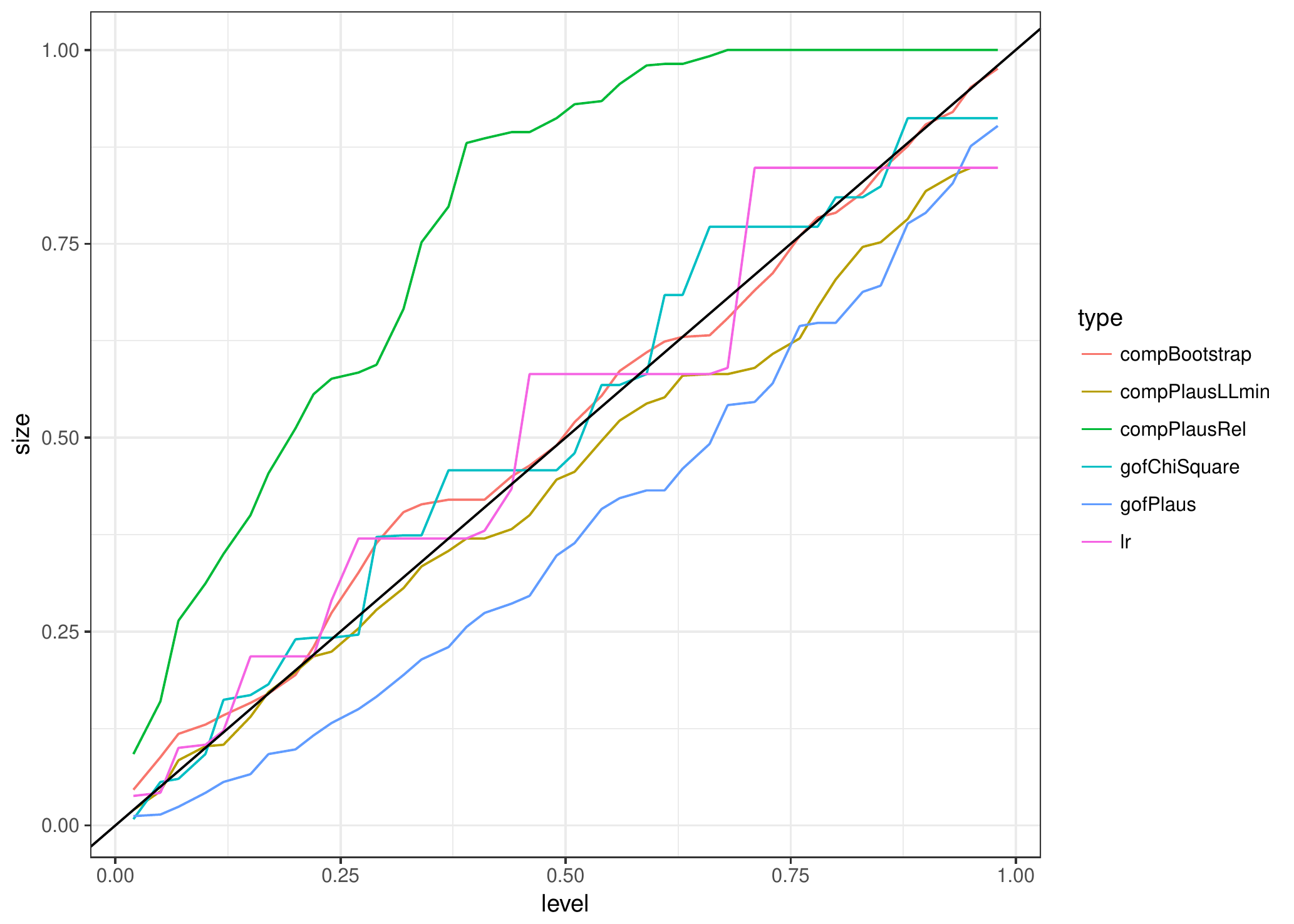}
	\end{center}
	\caption{Simulation results under the null hypothesis for a samples size of $N=8$ and a binomial family. X-axis is the $\alpha$-level of the test and Y-axis is the actual test size. {\it Bootstrap, W-Plausibility LR, W-Plausibility Rel, LR} are tests performing a model comparison. {\it Chi-Square, Plausibility} evaluate consistency with a binomial model. For details, see text.}\label{fig:simN8}
\end{figure}

\begin{figure}
	\begin{center}
	\begin{tabular}{cc}
	\includegraphics[height=.35\hsize]{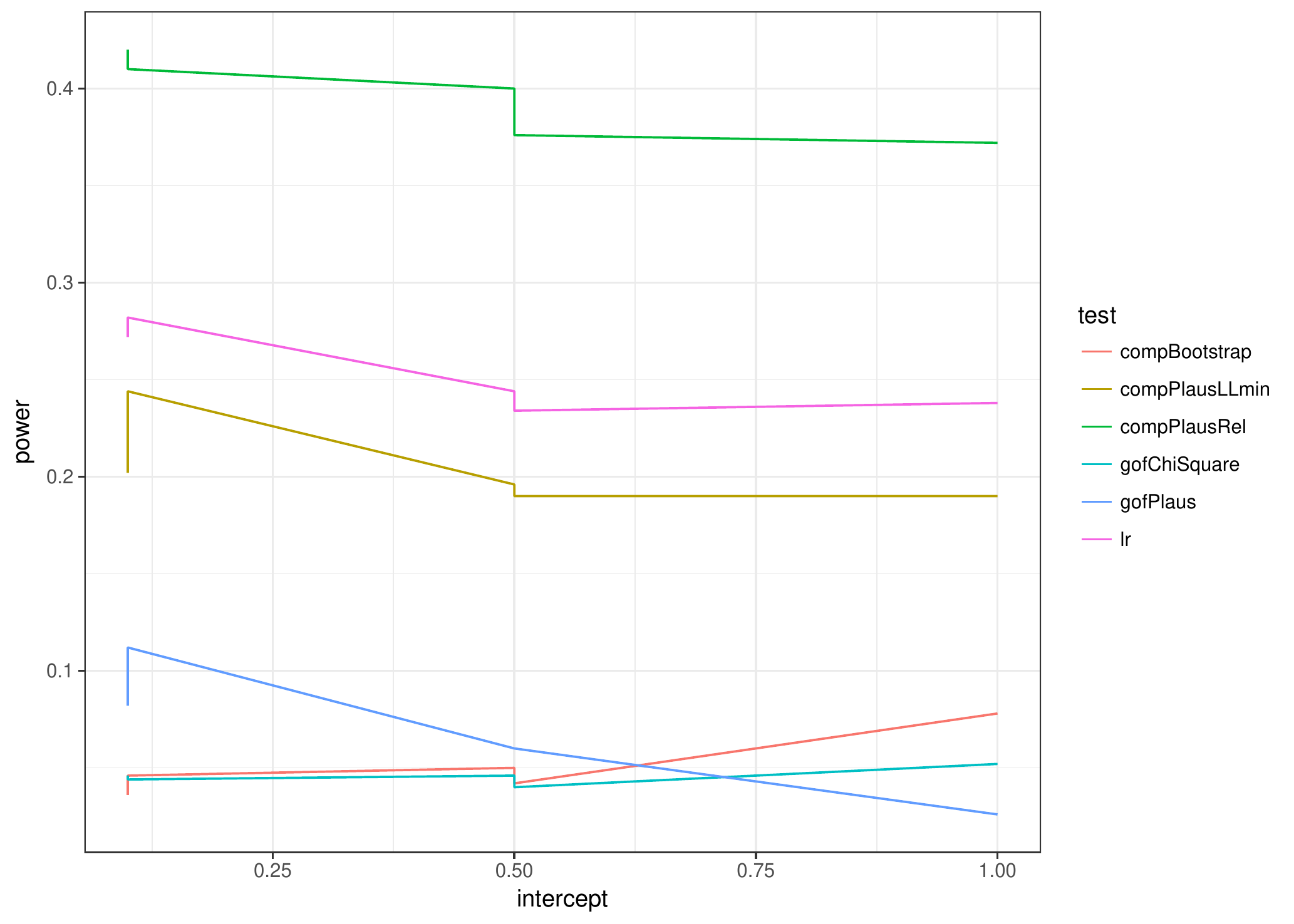} &
	\includegraphics[height=.35\hsize]{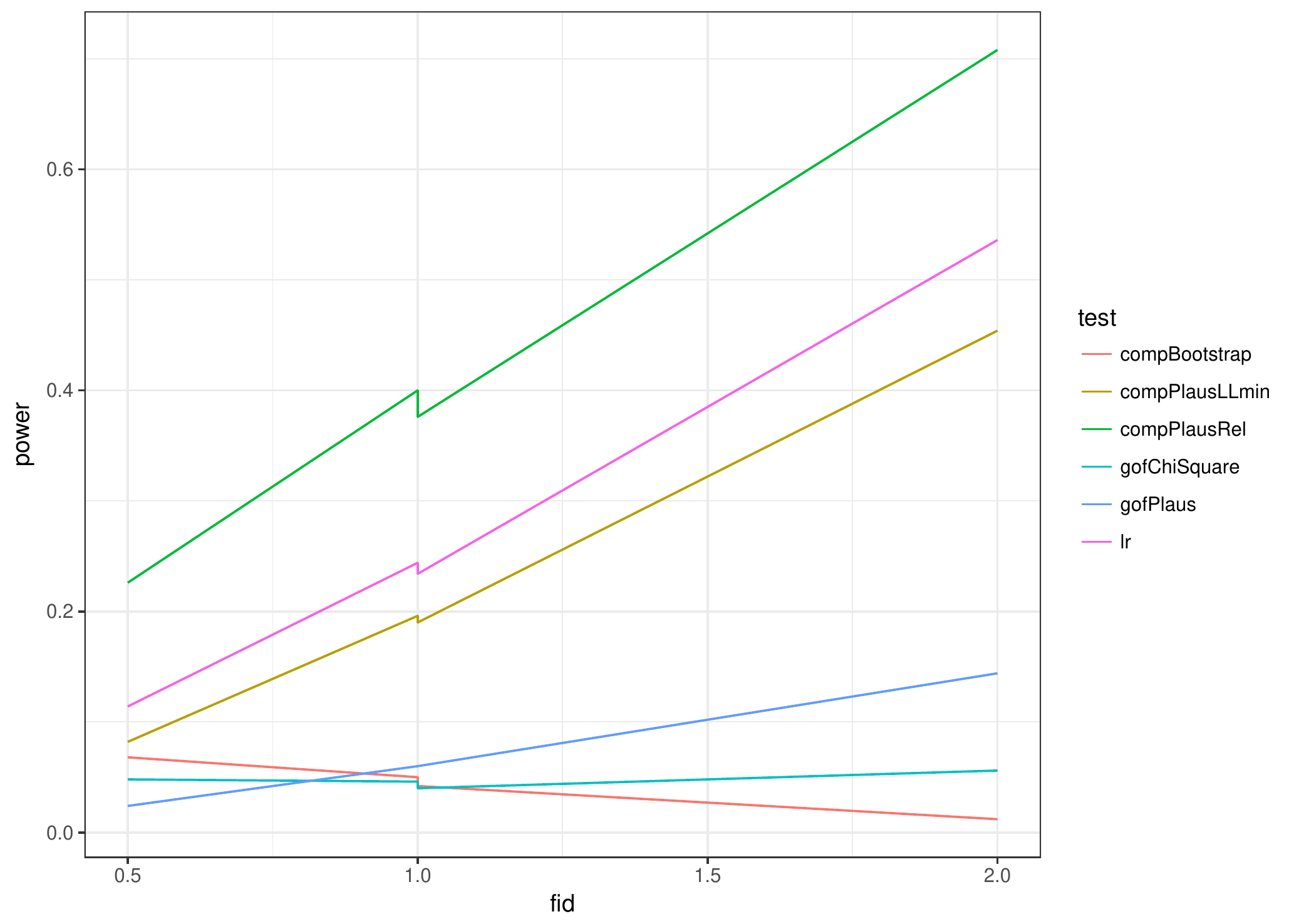}\\
	(A) & (B)\\
	\end{tabular}
	\end{center}
	\caption{Simulation results under the alternative hypothesis and a sample size of $N = 8$. (A) Fixed family effect (log-odds ratio 1) and varying intercept. (B) Fixed intercept (0.5) and varying family effect. X-axis represents the varying parameter value and the Y-axsis is power. Change points indicate scenarios that have been evaluated twice. For the compared testing procedures, see text.}\label{fig:simN8A}
\end{figure}

\subsubsection{Data analysis}

We used data from a larger database on Retinoblastoma collected from the literature. Initially, we selected the largest families as they should be most informative. To restrict the computational burden, the five smallest families have been selected from this subset. Outcome data was explicitly ignored when making this decision. Outcome distribution is shown in table \ref{tab:rb-desc}A for both total families and mutation carriers. Parent-of-origin ({\it i.e.} the sex of the transmitting parent) for mutation carriers is summarized in table \ref{tab:rb-desc}B. From this table it is apparent that parental origin strongly influences tumor status.\par
To be able to compute plausibility statistics, stochastic integration was used as complete iteration of all possible events was unfeasible. Also a grid search over all parameters was not possible as the grid increases exponentially with the number of parameters. R function {\it optim} with the Nelder-Mead algorithm was used to find plausibility estimates. The LR statistic was computed by fitting nested models and using the R function {\it anova} with the $\chi^2$ statistic. Results are shown in table \ref{tab:rb-result}. Heterogeneity between families could not be demonstrated (first half). Notably, the P-value of standard plausibility is much larger than the P-value of the weighted plausibility. This reflects the fact that standard plausibility rejects against a wider class of alternatives whereas weighted plausibility focuses power on a small class of alternatives. Analysis of the parent-of-origin (POO) effect shows statistically significant findings for weighted plausibility and the LR test. The P-value for standard plausibility is not significant. Plausibility estimates for family effects are close to but not identical to ML estimates. Technically, the ML-estimates are corrected for POO whereas plausibility estimates are not which is one explanation of discrepancies apart from differences in methodology. In all cases, plausibility P-values are larger than LR-based P-values.

\begin{table}
	\begin{center}
	\begin{tabular}{cc}
		\begin{tabular}{l|r|rrr|rrr}
		Family & \# & Y0 & Y1 & Y2 & Y$_m$0 & Y$_m$1 & Y$_m$2 \\
		\hline
		1	& 17 & 13  & 3 & 1	& 3 & 3 & 1\\ 
		2	& 18 & 17  & 1 & 0	& 7 & 1 & 0\\
		3	& 19 & 16  & 1 & 2	& 5 & 1 & 2\\
		4	& 20 & 16  & 2 & 2	& 4 & 2 & 2\\
		5	& 31 & 20  & 9 & 2	& 6 & 8 & 1\\
		\end{tabular}
		&
		\begin{tabular}{l|rr}
		Y & pat & mat \\
		\hline
		0	& 13 & 12\\ 
		1	& 4  & 11\\
		2	& 1	 & 5
		\end{tabular}
		\\
		(A) & (B)
	\end{tabular}
	\caption{Descriptive data analysis. (A) Distribution of number of affected eyes. {\it \#}: number of family members, {\it Y$i$}: number of family members with $i$ affected eyes. {\it Y$_mi$}: numbers among mutations carriers. (B): cross-tabulation of eye affection status Y and parental origin of mutation ({\it pat}: paternal, {\it mat}: maternal).
	}\label{tab:rb-desc}
	\end{center}
\end{table}

\begin{table}
	\begin{center}
	\begin{tabular}{lllrrrr}
	Method & F0 & F1 & Intercept & Fid & Poo & P\\
	\hline
	Plausibility			& y $\sim$ 1  & y $\sim$ fid  & -2.18 & - & - & 0.447\\ 
	Weighted Plausibility	& y $\sim$ 1  & y $\sim$ fid  & -2.17 & - & - & 0.197\\ 
	LR						& y $\sim$ 1  & y $\sim$ fid  & -0.79 & -1.91, 0.28 & - & 0.171\\ 
	\hline
	Plausibility			& y $\sim$ fid  & y $\sim$ fid + poo & -0.75 & -2.10, 0.23 & - & 0.484\\ 
	Weighted Plausibility	& y $\sim$ fid  & y $\sim$ fid + poo & -0.63 & -2.69, 0.18 & - & 0.0028\\ 
	LR						& y $\sim$ fid  & y $\sim$ fid + poo & -1.63 & -2.83, 0.21 & 1.87 & 0.0019\\ 
	\end{tabular}
	\caption{Results of data analysis. {\it Method:} test used, {\it F0:} Null-model in notation outcome $\sim$ covariates, {\it fid}: factor for family, {\it poo}: parent-of-origin. {\it F1:} Model under alternative. {\it Intercept:} estimated coefficient. {\it Fid:} range of coeffiencts for the families. {\it Poo:} coefficient for parent-of-origin. {\it P:} P-value.}\label{tab:rb-result}
	\end{center}
\end{table}

\subsection{High-dimensional data}

In this sub-section, we investigate finite sample properties of plausible model comparisons in the high-dimensional setting, using simulations with a normally distributed outcome. We also apply the global test constructed above to a well-known prostate cancer data set where a binary outcome is modeled with a logistic model.

\subsubsection{Simulations}

\begin{figure}
	\begin{center}
	\begin{tabular}{cc}
	\includegraphics[height=.35\hsize]{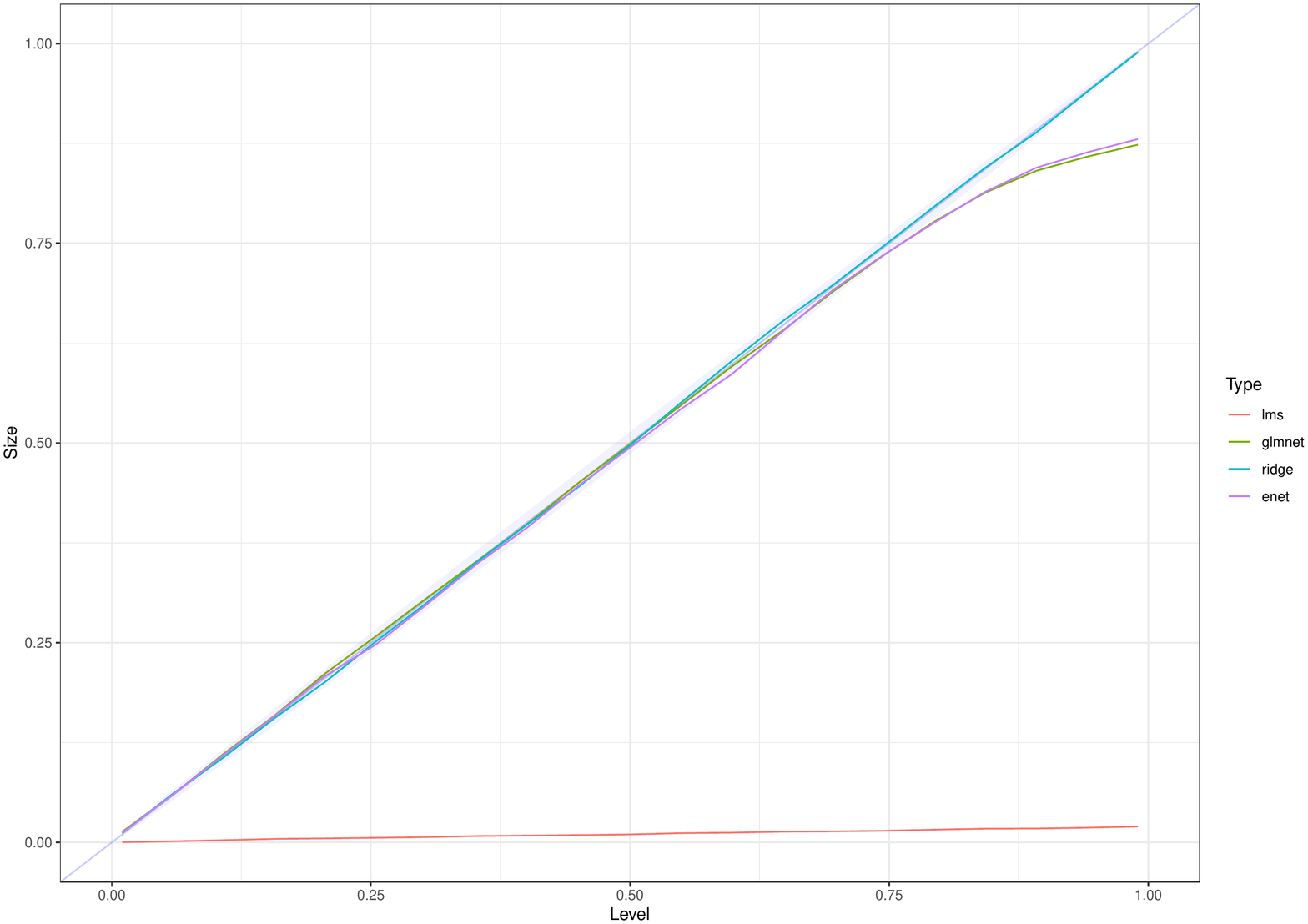} &
	\includegraphics[height=.35\hsize]{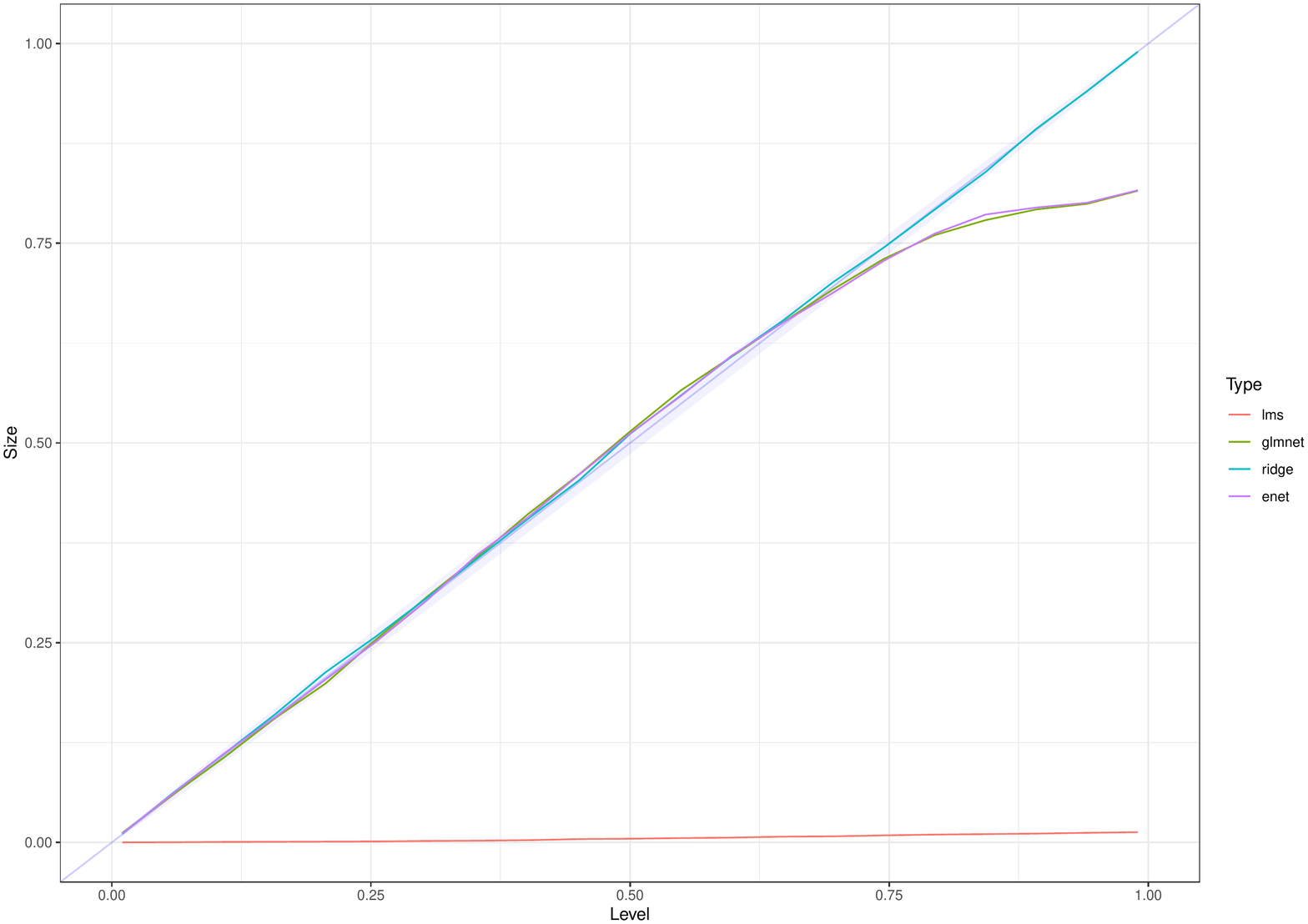}\\
	(A) & (B)\\
	\end{tabular}
	\end{center}
	\caption{Simulation results under the null hypothesis. X-axis is the $\alpha$-level of the test and Y-axis is the actual test size. (A) N = 200, p = 500, correlation low (see text), (B) N = 500, p = 500, correlation high (see text).}\label{fig:simPenalizedNull}
\end{figure}

In order to evaluate behavior of the compared tests, high-dimensional data was simulated. Sample size was chosen to be either $N = 200$ or $N = 500$. $p = 500$ covariates were simulated. Covariates were drawn in independent blocks of 10 covariates with an exchangeable correlation structure of $.1$ (low) or $.9$ (high). Under the null, the outcome was independently drawn from a standard normal distribution. $5 \times 10^3$ replications were used to determine test size. For the stochastic integration $10^3$ samples were used and the mixing parameter of elastic net regression was set to $\alpha = 0.9$. 1e3 data splits were used for {\it lms}.
Figure \ref{fig:simPenalizedNull} shows test sizes. Ridge regression perfectly exhausts the $\alpha$ level, whereas elastic net and Lasso exhaust the $\alpha$ level up to 0.75 for low correlation above which the procedures become conservative. This is due to the sparsity induced by these methods. For high correlation, this behavior is a bit more pronounced and conservative behavior starts at an $\alpha$ level of roughly 0.6. Under the null, elastic net and Lasso behave almost identical. {\it lms} shows poor exhaustion of the $\alpha$-level in both scenarios.

\begin{figure}
	\begin{center}
	\begin{tabular}{c}
	\includegraphics[height=.65\hsize]{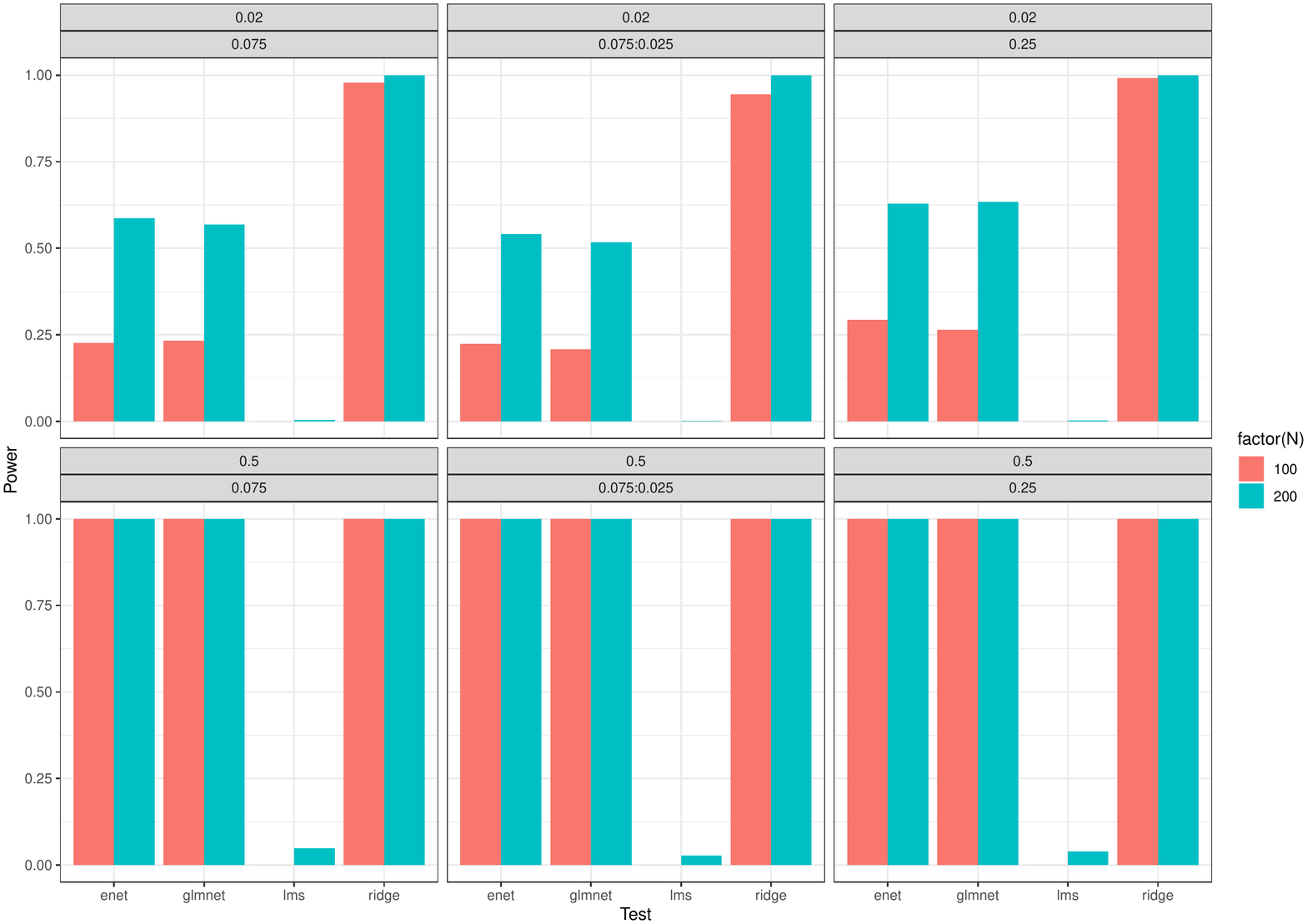}\\
	(A)\\
	\includegraphics[height=.65\hsize]{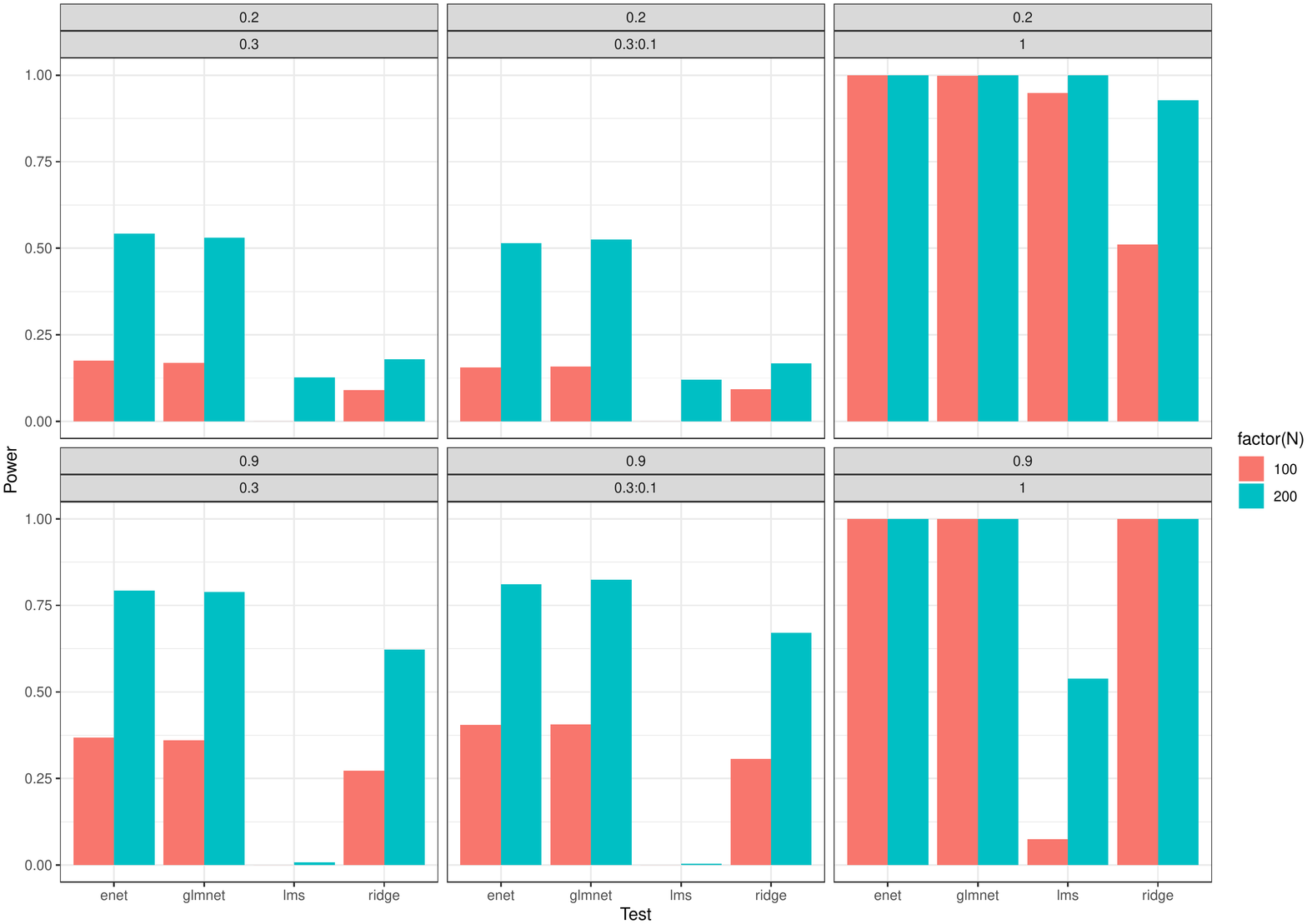}\\
	(B)
	\end{tabular}
	\end{center}
	\caption{Simulation results under the alternative hypothesis. Color indicates sample size. The top row of each cell indicates correlation structure (0.02, 0.5, see text). The second row indicates range of effect sizes used for the first two blocks of covariates (see text). Part (A): dense alternative; Part (B): sparse alternative. Methods considered are Lasso ({\it glmnet}), Elastic Net ({\it enet}), Lasso multi-split ({\it lms}), Ridge regression ({\it ridge}).}\label{fig:simPenalizedAlt}
\end{figure}

Under the alternative, simulations of covariates were performed identically to simulations under the null hypothesis. Sample sizes of $N = 100$ and $N = 200$ were considered. To generate an outcome, positive regression coefficients were chosen for the first two blocks of covariates. For the dense scenario (A), regression coefficients of the first or the first two blocks were set to identical values within each block. The sparse scenario (B), only assigns positive regression coefficients to the first covariate of the first or the first two blocks. All other coefficients were zero. Outcomes were simulated by adding a standard normal to the predictor $\mathbf X \mathbf \beta$. Results from the simulations are shown in Figure \ref{fig:simPenalizedAlt}. In the dense scenario (A), ridge regression performed best throughout. Increasing correlation increased power for all methods. {\it lms} has poor power in these scenarios.\par
In the sparse scenario, lasso and elastic net performed very similarly and were the most powerful procedures in all scenarios that were considered. {\it lms} could outperform ridge regression for the scenario of a single, strong effect and low correlation between covariates. In all other scenarios, {\it lms} was the least powerful procedure. In general, power was again poor for {\it lms}.\par
For a sample size of 100, plausibility based methods had sufficient power in scenarios that matched the method (dense vs sparse). For a sample size of 200 power was still below 80\% for many scenarios.\par
In general, correlation structure was very important as power increases substantially when comparing low and high correlation scenarios, {\it e.g.} power for lasso and an effect size of 0.075 for low correlation and sample size of 200 is $\sim$55\% and increases to $>$ 80\% (dense scenario).

\subsubsection{Data analysis}

As an illustration, we analyze a prostate cancer data set \citea{singh_gene_2002} as provided by R package {\it sda} \citea{ahdesmaki_sda:_2015}. The data set contains healthy (N = 50) and prostate cancer samples (N = 52) and measurements of 6033 gene expression values. We analyze the data set using a logistic model and the penalties used in the simulation section as well as the {\it lms} method. To get the {\it lms} run, the penalty parameter had to be increased from the default choice ($\lambda = \sqrt{N + p}/5$ instead of $\lambda = \sqrt{N + p}/10$). With this penalty, {\it lms} could not select any predictor and resulted in a P-value of 1. All plausibility models resulted in P-values $< 10^{-3}$, when the P-value was limited by the number of stochastic integration samples.

\begin{figure}
	\begin{center}
	\includegraphics[width=.9\hsize]{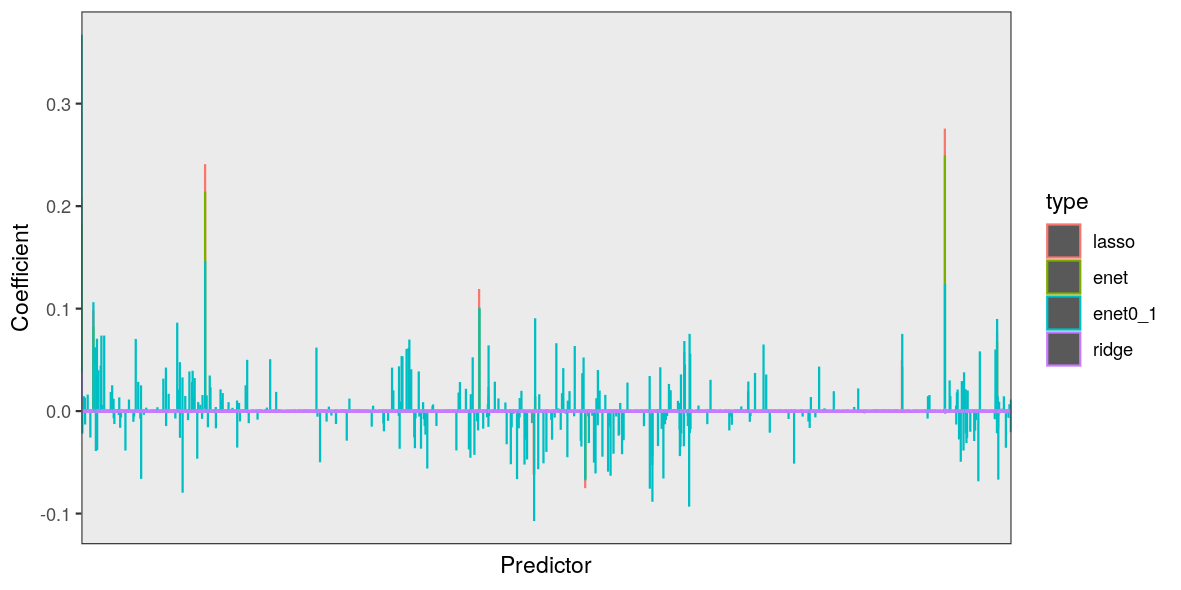}
	\end{center}
	\caption{Data analysis. X-axis are predictors 1, ..., 6033 with regression coefficients represented as bars for the methods: Lasso ({\it lasso}), Elastic Net ({\it enet}, $\alpha = 0.9$), Elastic Net ({\it enet0\_1}, $\alpha = 0.1$), Ridge regression ({\it ridge}).}\label{fig:dataAnalysis}
\end{figure}

To illustrate the methods, figure \ref{fig:dataAnalysis} shows the regression coefficients from the penalized plausibility models evaluated at the median penalty parameter from models generated during stochastic integretion. The figure clearly reflects the different sparsities of the methods. While lasso and elastic net ($\alpha = 0.9$) select few variables, elastic net with $\alpha = 0.1$ selects more and ridge all variables. Effect size change correspondingly (large for sparse methods, low for dense methods). We discuss this finding later.

\section{Discussion}

%
%

In this paper, we have extended the plausibility framework by a weighing component. For discrete data, the weighing leads to a re-ordering of data sets so that cumulative probabilities are no longer evaluated according to ordered probabilities but according to an (arbitrary) re-ordering induced by the weighing. Intuitively, it is clear that properties from plausibility carry over to weighted plausibility as long as the weighing does not depend on the data. Ordering by probability thus turns out to be just one possible ordering corresponding to goodness-of-fit evaluations. Comparing models corresponds to weighing by LR. More precisely, a plausibility-ratio should be considered but this is computationally expensive. This aspect will be further discussed below.

The flexibility of weighing is illustrated by our data analysis in the high-dimensional setting. Regression coefficients in this case strongly depend on the sparsity of the method. This fact is well known among practitioners and can lead to difficulties in model interpretation as it is often unclear how to choose sparsity. To the authors' knowledge theoretical underpinning is lacking. A weighted plausibility can be constructed for which the weighing function tunes the sparsity parameter. This would create a procedure reflecting steps taken in practice with guarantees for the alpha-level. Other weighting schemes could consider non-nested models. By evaluating data under one of the models and weighing data sets by a contrast between the models. This could be prediction accuracies, cross-validated LRs, or information criteria. The possibilities are clearly endless.\par

On the other hand, the resulting tests have less strong interpretations as compared to alternative approaches. In the high-dimensional data analysis, there are some important difference between the considered plausibility tests and the {\it lms} method. {\it lms} uses models coming from a conditional lasso model. {\it lms} can thereby reject a single co-variate while controlling family-wise error rates. The only test we considered was a global test based on the linear predictor of all covariates. Rejecting the null hypothesis would therefore not entail the rejection of any single covariate and further steps are needed. One approach is to hierarchically split the covariates according to some outcome-independent procedure (say hierarchical clustering), and test covariates along the tree. This can lead to efficient procedures, {\it e.g.} \citea{meinshausen_hierarchical_2008,goeman_multiple_2008}. A straightforward approach would analyze covariates marginally, {\it i.e.} the model would be an intercept model against the model of all selected variables and not be corrected for the other covariates. If a conditional model is desired, model comparison of penalized models can be used via a Bayesian prior in the analysis. The hyper-parameters would have to be fitted as part of the optimization in an empirical-Bayes spirit. However, it is very expensive to fit such a plausibility model as in general a grid search is required. It is an open question in how far approximations can be used and whether errors can be bounded, if say, ML-estimates are used instead of plausibility estimates.\par

The plausibility framework does have some important limitations. We mention non-plausible parameters and non-plausible parameter values. Non-plausible parameters have to be handled by a different estimation procedure which might involve iterated estimation between say, plausibility and ML. In the normal model, this limitation is non-essential. Non-plausible parameter values seem to be rather a technical problem, although the problem manifested itself in simulations during the preparation of this manuscript. Some care therefore needs to be taking when deriving the likelihood to be used. A major limitation is certainly computation time. While this limitation is shared with other methods like cross-validation as used in targeted learning, bootstrapping, or data-splitting ({\it lms}), the problem is usually more severe for plausibility. During stochastic integration, a full penalized analysis including parameter tuning through cross-validation has to be performed for every data set. In our analyses of high-dimensional data, we did not include covariates under the null, which made the evaluation of plausibility relatively cheap. A single data set was analyzed in a matter of a few minutes. There are several ways to improve efficiency. Using importance sampling during stochastic integration seems to be a promising approach. We have also implemented some short-cuts, for example evaluating the penalty parameter first and using it throughout stochastic integration. Certainly, efficiency remains a challenge.\par

Plausibility can guarantee test sizes under the null for finite samples. A fully non-parametric treatment leads back to a non-parametric bootstrap procedure as data would be drawn from the empirical distribution function (EDF). Every data set has maximal probability under its own EDF and the plausibility approach would therefore not be useful. A compromise by using, say, kernel density estimates could mitigate this problem. However, test sizes cannot be guaranteed under such an approach as the density estimate would involve an estimation error. Further research is needed to explore non-parametric or semi-parametric models.

Targeted analysis defines so-called target parameters for which statistical inference is required. Typically, this would be the difference between a treatment effect and a counterfactual opposite treatment decision in the case of clinical studies \citea{van_der_laan_targeted_2011}. Testing the target parameter corresponds to a model comparison of a null effect with a model including a treatment effect. Given that the target parameter can be defined using high-dimensional data, the model comparison is high-dimensional. The proposed way to evaluate the global test is the influence function using cross-validation \citea{zheng_asymptotic_2010}. Weighted plausibility is a natural alternative to evaluate the test statistic in targeted analyses. Such an analysis would be similar to what was used in section \ref{section:high-dim}. Weighted plausibility is therefore an alternative to perform statistical inference for the ensembles learned in targeted learning. Such an analysis could be beneficial in some cases as the full data is used as compared to cross-validation as proposed for targeted learning. Empirical studies would have to show under which circumstances plausibility would offer advantages or would be disadvantageous.\par

In conclusion, the plausibility framework allows to conduct exact model comparisons in small sample size situations such as our RB data set. In these cases, valid concerns can be present about the validity of asymptotic or empirical p-values computed by other means, which was confirmed by our simulations. The same principles can be applied in the high-dimensional setting enabling statistical inference in these situations. The conceptual similarity to the bootstrap is reflected by the fact that any data independent test statistic can be used as a weighing function allowing further applications in high-dimensional data analysis, targeted analysis or non-nested model comparisons.

\par

{\it Software}\par
We provide an implementation for analyzing data using weighted plausibility. The use of this package is documented in  appendix \ref{app:r}.

\newpage
\appendix
\section{Appendix}

\subsection{Non-plausible parameters} \label{app:nonplausible}

If we take the likelihood for a i.i.d Bernoulli experiment to be
	$ L(\theta; Y) = \prod_{i = 1}^N \theta^{Y_i} (1 - \theta)^{1 - Y_i} $,
	$ L(0.5, Y) = .5^N $.
Therefore, $\pl_Y( [0, 1] ) = 1$ and the supremum is achieved in $\theta = .5$ and $\theta = \hat \theta$. For this likelihood, every data set would have plausibility of 1 even if we exclude $\theta = .5$, and the non-plausible value of 0.5 would not be important, however, the example can be easily extended.\\


If we iterate the experiment and draw $K$ groups of $N$ i.i.d Bernoulli variables, the non-plausible value starts to matter. The likelihood can be modeled as
	$L(\theta; Y) = \prod_{j =  1}^K \prod_{i = 1}^N \theta_j^{Y_{ji}} (1 - \theta_j)^{1 - Y_{ji}}$.
If we choose
	$ A = \{\theta \in [0, 1]^K | \theta_j = \theta_1 \, \forall j \} $, 
the likelihood reduces as above for $\theta_1 = 0.5$ and $\pl_Y(A) = 1$ for all data sets. However, the plausibility restricted under $A$ is smaller than 1 in general, if we exclude $\theta_1 = .5$ from $A$ and the $\theta_i$'s are truly different. If we now consider the nested experiment to come from a binomial distribution, where $Y'_j$ is the number of outcomes 1 in each block, we get
	$L'(\theta; Y') = \prod_{j =  1}^K { N \choose Y'_j } \theta_j^{Y'_{j}} (1 - \theta_j)^{N - Y'_{j}}$.
In this case, 0.5 is no longer a non-plausible value. For example, $K = 2$ and $Y'_1 = 0, Y'_2 = N$, results in $L'(0.5, Y') = 0.5^{2N}$. However,  $Y''_1 = N/2, Y''_2 = N/2$ (assuming $N$ even), results in
	$L'(0.5, Y'') = { N \choose N/2 }^2 0.5^{2N} > L(0.5, Y')$ for $N \ge 2$.
$ \theta_1 = 0.5 $ is therefore no longer a non-plausible value. However, $L'$ still does not have unique point mass, as for $\theta_1 = .5$, we can replace $Y'$ with $Y''' := N - Y'$ and get $L(\theta_1; Y''') = L(\theta_1; Y')$ for, in general, non-exchangeable observations. The assumption requiring unique point-mass is therefore too strong as the plausibility function can be readily evaluated. We hypothesize that it is sufficient to require unique point-mass on a dense subset of the parameter space although this is not straightforward to show.

\subsection{Plausibility is a special case of weighted plausibility}\label{app:specialcase}

To show that that plausibility is a special case of weighted plausibility, we choose the weighting function as $w(y) = pl_y(A)$. For observation $y$, we then get $pl_y^w(A) = \sup_{θ \in A} F_θ(T^w_{y, θ})$ and $F_θ(T^w_{y, θ})$ becomes
\begin{align*}
	&	 F_θ(pl_y(A)) \\
	= &  P_θ(pl_Y(A) ≤ pl_y(A)) \\
	= &  P_θ(F_{θ^{*'}}(T_{Y, θ^{*'}}) ≤ F_{θ^*}(T_{y, θ^*})) \\
	≤ &  P_{θ^*}(F_{θ^{*'}}(T_{Y, θ^{*'}}) ≤ F_{θ^*}(T_{y, θ^*})) \\
	≤ &  P_{θ^*}(F_{θ^{*}}(T_{Y, θ^{*}}) ≤ F_{θ^*}(T_{y, θ^*})) \\
	\stackrel{(1)}{=}
	  &  P_{θ^*}(T_{Y, {θ^{*}}} ≤ T_{y, θ^{*}}) = F_{θ^*}(T_{y, θ^*}) = \mathrm{pl}_y(A).
\end{align*}
Here, $θ^{*}, θ^{*'}$ are the plausibility estimates for $y$ and $Y$, respectively.
The distribution function $F_{θ^*}$ can be dropped in (1), as $F_{θ^*}(T_{Y, θ^*}) ≤ F_{θ^*}(T_{y, θ^*})$ iff $T_{Y, {θ^*}} ≤ T_{y, θ^*}$.
The second inequality follows due to $F_{θ^{*}}(T_{Y, θ^{*}}) ≤ F_{θ^{*'}}(T_{Y, θ^{*'}})$.
The first inequality is shown indirectly. First, we note that we evaluate the event $Z := \{Y | pl_Y(A) ≤ pl_y(A)\}$, which is independent of $θ$, once under $P_θ$ and once under $P_{θ^*}$. Let us assume that there exists a $θ$ so that $P_θ(Z) > P_{θ^*}(Z)$. Assuming that $T$ is either discrete or continuous (but not mixed), we can find an outcome $z := \argmax_{ z' ∈ Z } T_{z', θ}$. But then,
	$pl_z(A) \ge F_θ(T_{z, θ}) > pl_y(A)$, as
	$\{Y | T_{Y, θ} ≤ T_{z, θ} \} \supset Z$. This is a contradiction to the definition of $Z$.

Supremizing over $θ$ shows that $\sup_θ F_θ(pl_y(A)) \ge \mathrm{pl}_y(A)$ as $θ^*$ is among the $θ$'s, which concludes the proof.

Note, that we do not need the plausibility estimates $θ^{*}, θ^{*'}$ or $z$ to be unique. They only need to achieve the supremum or maximum in the respective terms.

\subsection{Proofs}

\subsubsection{Proof of lemma \ref{lemma:conv}}

\begin{proof}
	If $\mathrm{Pl}_\theta(\alpha)$ is continuous in $\alpha$, {\it i.e.} $\mathrm{Pl}_\theta(\alpha) = \mathrm{Pl}_\theta(\alpha-)$, $\mathrm{Pl}_\theta(\alpha) = \alpha$ by the properties of the CDF. Otherwise, assume $\theta$ known and define $pl_{Y, \theta} := F_\theta(T_{y, \theta})$ and $Pl_{Y, \theta}$ the corresponding CDF. Then $\Delta := \sup\{\mathrm{Pl}_{1, \theta}(\alpha) - \mathrm{Pl}_{1, \theta}(\alpha-)\}$ the supremum over the discontinuities of $\mathrm{Pl}_{1, \theta}(\alpha)$ which we assume to be bound away from 1.
	Let $\alpha_M := \arg \sup\{\mathrm{Pl}_{n, \theta}(\alpha) - \mathrm{Pl}_{n, \theta}(\alpha-)\}$ and $Y_m := \{y | \pl_{y, \theta}(A) = \alpha_m\}$. Then for each $y_m \in Y_m, T_{y, \theta} \le \Delta^n$, as $T = T_n$ is the product measure of $T_1$ and $P_\theta(Y_m) \le |Y_m| \Delta^n$ ($|\cdot|$ denotes cardinality). Due to the uniqueness assumption of point masses, each $y \in Y_m$ contains exchangeable observations, {\it i.e.} different vectors $y$ are identical up to ordering. The size of $|Y_m|$ is given by a multinomial coefficient which can be upper bounded by the binomial coefficient $\binom{n}{n/2}$. The maximal discontinuity for this case is achieved for class probabilities close to .5 and therefore $\Delta \approx 0.25$. By Sterling's approximation $\binom{n}{n/2} \sim \frac{1}{\sqrt{\pi/2 n}} 2^n$, so that $P_\theta(Y_m) \to 0$ with rate of at least $\sqrt{n}$. This implies that $\mathrm{Pl}_{n, \theta}(\alpha)$ converges pointwise to the CDF of the uniform. Applying Portmanteau's theorem completes the proof for known $\theta$.\par
	Let now $\theta^*_n$ be a sequence of plausibility estimates with $\theta^*_n \stackrel{P}{\to} \theta^*$ (see proof of theorem 3 in\citea{martin_plausibility_2015}), or equivalently, $ \pl_n(\theta) \to 0$ with $P_{\theta^*}$-probability for $\theta \ne \theta^*$.  $\alpha_M(\theta)$  can therefore be bounded by a continuous function in an appropriately chosen neighborhood $U$ of $\theta^*$, which completes the proof.
\end{proof}

\subsubsection{Lemma \ref{lemma:symm}}

To preapre the proof of lemma \ref{theorem:lrEquiv}, we start with a lemma about event sequences.

\newtheorem{lemmaSymm}[lemmas]{Lemma}
\begin{lemmaSymm}\label{lemma:symm}
	Under the assumptions of section \ref{section:comparison}, let $E_1, ..., E_n$ and $E'_1, ..., E'_n$ be two sequences of events for which either $E_n \subset E'_n$ or $E'_n \subset E_n$. Let $D_n = E_n \triangle E'_n$ be the symmetric difference. If probabilities of both sequences converge for some $\theta \in \Theta$, $\alpha := \lim_{n \to \infty} P_\theta(E_i), \beta := \lim_{n \to \infty} P_\theta(E'_i)$, then
	$$
		\lim_{n \to \infty} P_\theta(E_i) = \lim_{n \to \infty} P_\theta(E'_i)
		\quad\mathrm{if\, and\, only\, if}\quad
		P_\theta(D_n) \to 0 \,\mathrm{in\, probability.\,}
	$$
\end{lemmaSymm}

\begin{proof}
	,,$\Rightarrow$'': Assume $P_\theta(D_n) > \epsilon > 0$ for $n > n_0 \in \mathbb{N}$, $U_n := E_i \cup E'_i$, $I_n := E_i \cap E'_i$.\\
	If $E'_n \supset E_n$ for all $n > n_1$, then $P_\theta(E'_n) = P_\theta(U_n) =  P_\theta(I_n) + P_\theta(D_n) = P_\theta(E_n) + P_\theta(D_n) \underset{n \to \infty}{\to} c > \alpha + \epsilon$.
	Otherwise, $E_i \supset E'_i$ infinitely often. By applying the above argument to the sub-sequence for which this inclusion holds, the same contradiction arises.\\
	,,$\Leftarrow$'':
	Assume $|\beta - \alpha| > \epsilon > 0$. For $E'_n \supset E_n$ for all $n > n_1$,  $P_\theta(E'_n) = P_\theta(U_n) =  P_\theta(I_n) + P_\theta(D_n) \underset{n \to \infty}{\to} \alpha$, a contradiction, which arises again for the sub-sequence for which $E_i \supset E'_i$.
\end{proof}

\subsubsection{Proof of theorem \ref{theorem:lrEquiv}}

\begin{proof}
	$Y = (Y_1, ..., Y_n)$, $Y_i \ \mathrm{iid} \sim P^0_\theta$, $P_\theta = (P^0_\theta)^n$. The rejection region of the LR test is given by $R^{LR}_n = \{y | w(y) < c \}$, where $c$ is the appropriately transformed quantile of a $\chi^2$-distribution.
	For the rejection region of the plausibility test, we have $R^{pl}_n = \{y | w(y) < c'\}$, where $c'$ is chosen smallest, so that $P_\theta(R^{pl}_n) \le \alpha$.
	By definition, we therefore have $R^{pl}_n \subset R^{LR}_n$ or $R^{LR}_n \subset R^{pl}_n$. Let $D_n$ be the symmetric difference between the rejection regions, {\it i.e.} $D_n = R^{pl}_n \triangle R^{LR}_n$. From standard likelihood theory we have $P_\theta(R^{LR}_n) \overset{P}{\to} \alpha$. Using lemma \ref{lemma:conv}, we also have  $P_\theta(R^{pl}_n) \overset{P}{\to} \alpha$. Applying lemma \ref{lemma:symm} completes the proof.
\end{proof}

For a given $\theta$, the rejection region of the LR test is composed of $y$'s, for which the LR is large, {\it i.e.} $w(y)$ is small. As $w(y)$ is also used as the weighing function in the plausibility test, the rejection regions overlap and thereby fulfill the conditions of Lemma \ref{lemma:symm}.

\subsubsection{Proof of lemma \ref{lemma:marginal}}

\begin{proof}
	For data generating parameter $\theta^\bullet = (\psi^\bullet, \lambda^\bullet)$ it is to be shown that 		$P_{\theta^\bullet}(
		\Pi^{w}_y(\alpha) \ni \psi^\bullet
	) \ge 1 - \alpha.
	$

	For event $\Pi^{w}_Y(\alpha) \ni \psi^\bullet$ and some appropriately chosen constant $c_\alpha \in \mathbb{R}$, we have\\
	\begin{align*}
		& \psi^\bullet  \in \Pi^{w}_Y(\alpha) = \{ \psi \, : \, \mathrm{ppl}_Y(\psi) > \alpha \} \\
	\Leftrightarrow \qquad
		&  \mathrm{ppl}_Y(\psi^\bullet) > \alpha\\
	\Leftrightarrow \qquad
		&  \sup_\lambda T_{Y, (\psi^\bullet, \lambda)} > \alpha\\
	\Leftrightarrow \qquad
		& Y \in \{ y \,:\, l(\psi^\bullet, \lambda^*) - l(\psi^*, \lambda^*) > c_\alpha \}\\
		& \supset \{ y \,:\, l(\psi^\bullet, \lambda^\bullet) - l(\psi^*, \lambda^*) > c_\alpha \} =: \mathcal{Y}_{c_\alpha}
	\end{align*}
	
	The inclusion in the last step holds, as $ l(\psi^\bullet, \lambda^*) \ge l(\psi^\bullet, \lambda^\bullet) $.  $c_\alpha$ is chosen to fulfill the constraint imposed by $\alpha$. The final event $\mathcal{Y}_{c_\alpha}$ is exactly the event for which the plausibility function for $\theta^\bullet$ has a value greater or equal to $1 - \alpha$. As the plausibility function is stochastically larger than uniform, computing probabilities of all events proofs the lemma.\\


\end{proof}

The inclusion relation between the last events reflects the price that is paid for profiling out $\lambda$ in the construction of the marginal plausibility region.

\subsubsection{Proof of lemma \ref{lemma:profileML}}

\begin{proof}
	Let $T^{\hat \sigma^2}_{Y, \beta}$ be the profile-plausibility and assume $\beta^* \ne \hat \beta$.
	Define data $Y'$ as having rescaled residuals by the factor $\hat \sigma^2(\beta^*) / \hat \sigma^2(\hat \beta)$ and being shifted by $\mathbf X (\hat \beta - \beta^*)$.
	Then, $
		L( Y, (\beta^* , \hat \sigma^2(\beta^*)) ) =
		L( Y', (\hat \beta, \hat \sigma^2(\hat \beta)) )$.
	As the likelihood for any observation is retained by this transformation, also $
		F^{\hat \sigma^2(\beta^*)}_{\beta^*}(T^{\hat \sigma^2(\beta^*)}_{Y, \beta^*}) =
		F^{\hat \sigma^2(\hat \beta)}_{\hat \beta}(T^{\hat \sigma^2(\hat \beta)}_{Y', \hat \beta})$.
	The distribution function $F^{\sigma^2}_{\beta}(t(\mathbf Y))$ is a radial function in $\beta$ and can be evaluated by the integral $1 - C \int_0^R \varphi_{\sigma^2}(r)r^{N - 1}dr$, where $\varphi$ is the density of the normal distribution with mean zero and variance $\sigma^2$, $C$ is a normalizing constant and $R = \|Y - \beta\|$ ($R = \varphi_{\sigma^2}^{-1}(t), t(\mathbf Y) = \varphi_{\sigma^2}(\|Y - \beta\|)$). Therefore, $F^{\sigma^2}_{\beta}(t)$ is strictly, monotonously decreasing in $t$ but $\| Y' - \hat \beta \| > \| Y - \hat \beta \|$, a contradiction.

\end{proof}

\subsubsection{Proof of lemma \ref{lemma:profile}}

\begin{proof}
	\begin{align*}
		\pl^{*w, \hat \sigma^2}_{Y, \beta} (A)
		= \int_{w(\mathbf Y') \le w(\mathbf Y)}
			T^{w}(\mathbf Y', \beta) dP^{\mathbf Y'|\beta^*},
	\end{align*}
	where $P^{\mathbf Y'|\beta^*}$ is the conditional distribution of $\mathbf Y$ given $\mathbf Y \in \{\mathbf Y'|\beta^*(\mathbf Y') = \beta^*\}$. Here, $\beta^*(\mathbf Z)$ denotes the plausibility estimate for $\mathbf Z$. Wlog let $\mathbf Y$ be standardized, so that the $\mathbf Y'$ in the integral above are uniform on the $S^{N - 1}$ sphere. We can now invoke lemma \ref{lemma:W} for the single distribution  $P^{\mathbf Y'|\beta^*}$. $\pl^{*w, \hat \sigma^2}_{Y, \beta} (A)$ is therefore stochastically larger than uniform for any profile-plausibility estimate.
\end{proof}

\subsection{R-package}\label{app:r}

We provide an R-package at \href{https://github.com/sboehringer/plausibility}{https://github.com/sboehringer/plausibility}. An example, calculating a model comparison for the R-data set {\em mtcars} is shown below.

\begin{verbatim}
	library('plausibility')
	data('mtcars')
	plGlm = new('plausibilityBinomial', am ~ 1, am ~ ., mtcars, Nsi = 1e3L);
	plausibility(plGlm, optMethod = 'optim');
\end{verbatim}

\bibliographystyle{unsrt}  
\bibliography{literature-sb}

\end{document}